\documentclass[10pt,a4paper]{amsart}
\usepackage[francais,english]{babel}
\usepackage{graphicx,amssymb,color,amscd}
\usepackage{tikz}

\theoremstyle{plain}
\newtheorem{theorem}{Theorem}[section]
\newtheorem{lemma}[theorem]{Lemma}
\newtheorem{proposition}[theorem]{Proposition}

\theoremstyle{definition}
\newtheorem{definition}[theorem]{Definition}
\newtheorem{thm}{Theorem}
\theoremstyle{remark}
\newtheorem{notation}[theorem]{Notation} 
 
\newtheorem{remark}[theorem]{Remark}
\newtheorem{ex}[theorem]{Example}

\newtheorem*{acknowledgments}{Acknowledgments}
\numberwithin{equation}{section}
\numberwithin{figure}{section}

\newcommand{\bd}{\begin{description}}   
\newcommand{\ed}{\end{description}} 
\newcommand{\ba}{\begin{array}}      \newcommand{\ea}{\end{array}} 
\newcommand{\bc}{\begin{center}}     \newcommand{\ec}{\end{center}} 
\newcommand{\be}{\begin{enumerate}}  \newcommand{\ee}{\end{enumerate}} 
\newcommand{\beq}{\begin{eqnarray}}  \newcommand{\eeq}{\end{eqnarray}} 
\newcommand{\beQ}{\begin{eqnarray*}} \newcommand{\eeQ}{\end{eqnarray*}} 
\newcommand{\bi}{\begin{itemize}}    \newcommand{\ei}{\end{itemize}}

\newcommand{\dessin}[2]{
  \vcenter{\hbox{\includegraphics[height=#1]{#2.pdf}}}}

\newcommand{\lk}{\textrm{lk}}

\begin{document} 
%
%
\title[Kontsevich-LMO, Conway and Casson-Walker-Lescop invariants]{Universal invariants, the Conway polynomial and the Casson-Walker-Lescop invariant}
\author[A. Casejuane ]{Adrien Casejuane} 
\address{Univ. Grenoble Alpes, CNRS, IF, 38000 Grenoble, France}
\author[J.B. Meilhan ]{Jean-Baptiste Meilhan} 
\address{Univ. Grenoble Alpes, CNRS, IF, 38000 Grenoble, France}
\email{jean-baptiste.meilhan@univ-grenoble-alpes.fr}

\begin{abstract} 
We give a general surgery formula for the Casson-Walker-Lescop invariant of closed $3$-manifolds, seen as the leading term of the LMO invariant, in a purely  diagrammatic and combinatorial way. This provides a new viewpoint on a formula established by C.~Lescop for her extension of the Walker invariant. 
A central ingredient in our proof is an explicit identification of the coefficients of the Conway polynomial as combinations of coefficients in the Kontsevich integral. 
This  latter result relies on general \lq factorization formulas\rq\, for the Kontsevich integral coefficients. 
\end{abstract} 


%

\thanks{The second author is partially supported by the project AlMaRe (ANR-19-CE40-0001-01) of the ANR}

\maketitle
\section{Introduction}

A.~Casson defined in 1985 an invariant of integral homology spheres, by counting conjugacy classes of irreducible $SU(2)$--representations of the fundamental group  \cite{Akbulut-McCarthy,AM}.  
The Casson invariant was extended, first to rational homology spheres by  K.~Walker \cite{Walker1992}, then to all oriented closed $3$-manifolds by C.~Lescop \cite{Lescop}, via surgery formulas. 
We denote by $\lambda_L$ this \emph{Casson-Walker-Lescop invariant}. 

In \cite{Le-Murakami-Ohtsuki}, T.~Q.~T.~Le, J.~Murakami and T.~Ohtsuki defined an invariant of closed oriented $3$-manifolds. 
This \emph{LMO invariant} is built from the Kontsevich integral \cite{Kontsevich1993} of a surgery presentation, i.e. a framed link in $S^3$.
The \emph{Kontsevich integral} of a framed $n$-component link takes values in a graded space of chord diagrams on $n$ circles, 
while the LMO invariant lives in a graded space of trivalent diagrams; 
the procedure for extracting the latter invariant from the former one relies on a family of sophisticated combinatorial maps $\iota_n$ that \lq\lq replace circles by sums of  trees\rq\rq.   
The Kontsevich integral is universal among $\mathbb{Q}$-valued Vassiliev invariants, in the sense that any such invariant factors through the Kontsevich integral. 
Likewise, the LMO invariant is universal among $\mathbb{Q}$-valued finite type invariants of rational homology spheres.
Both invariants admit purely combinatorial and diagrammatic definitions, although concrete computations are in general rather difficult. 

A striking result is that the leading term of the LMO invariant, i.e. the coefficient of the lowest degree trivalent diagram $\dessin{0.5cm}{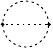}$, is up to a known factor the  Casson-Walker-Lescop invariant \cite{Le-Murakami-Murakami-Ohtsuki,Beliakova-Habegger}. 
This provides, in principle, a combinatorial procedure for computing the Casson-Walker-Lescop invariant from a surgery presentation, by computing the Kontsevich integral and keeping track of the coefficients of chord diagrams that produce a diagram $\dessin{0.5cm}{T}$ under the LMO procedure. 
This paper shows how this can be done completely explicitly, in terms of (classical) link invariants. 
Our first main result is as follows.
\begin{thm}[Thm.~\ref{Thm general}]\label{thm1}
Let $L$ be a framed oriented $n$-component link in $S^3$, 
and let $\mathbb{L}$ denote its linking matrix. 
Let $S^3_L$ be the result of surgery on $S^3$ along $L$. 
The  Casson-Walker-Lescop invariant $\lambda_L(S^3_L)$  is given by 
$$ \frac{(-1)^{\sigma_-(L)} \det \mathbb{L}}{8} \sigma(L) 
+ (-1)^{n+\sigma_-(L)}\sum_{k=1}^{n} \sum_{\substack{I \subset \{ 1, \ldots, n \} \\ |I| = k}} (-1)^{n - k} \det \mathbb{L}_{\check{I}} \mu_k(L_I),$$\\[-0.5cm]
where 
\begin{itemize}
 \item $\mathbb{L}_{\check{I}}$ is the matrix obtained from $\mathbb{L}$ by 
deleting the lines and column indexed by a subset $I$ of $\{ 1, \ldots, n \}$,
 \item $\sigma_+(L)$ and $\sigma_-(L)$ denote, respectively, the number of positive and negative eigenvalues of $\mathbb{L}$, and 
       $\sigma(L)=\sigma_+(L)-\sigma_-(L)$, 
 \item $\mu_k$ is a $k$-component framed link invariant which is explicitly determined by the coefficients of $\mathbb{L}$ and the Conway polynomial.  
\end{itemize}
\end{thm}
We do not give here the general explicit formula for the invariant $\mu_k$, which is postponed to Theorem \ref{cor:mun}. 
Let us only give here the formulas for the first two of these invariants: for a framed knot $K$ we have 
$ \mu_1(K) = \frac{1}{24} fr(K)^2 - c_2(K) + \frac{1}{12}$,   
and if $L=K_1\cup K_2$ is a $2$-component framed link, then $\mu_2(L)$ is given by 
$$ \frac{1}{12} \lk(L)^3 + \frac{fr(K_1) + fr(K_2)}{12} \lk(L)^2 + \lk(L)\Big( c_2(K_1) + c_2(K_2) - \frac{1}{12} \Big) - c_3(L), $$
where $c_k$ denotes the coefficient of $z^k$ in the Conway polynomial. 
These two invariants are involved in the case $n=2$ of Theorem \ref{thm1}, which recovers a result of S.~Matveev and M.~Polyak \cite[Thm.~6.3]{Matveev-Polyak} 
for the Casson-Walker invariant of rational homology spheres; see Remark \ref{rem:cestcadeaucamfaitplaisir} for details.

\begin{remark}
Theorem \ref{thm1} recovers the third global surgery formula of Lescop \cite[Prop.~1.7.8]{Lescop}, 
when restricted to integral surgery coefficients (see Remark \ref{rem:lescoop}). 
The proof relies centrally on the \lq key result\rq\, connecting the leading term of the LMO invariant and the Casson-Walker-Lescop invariant \cite{Le-Murakami-Murakami-Ohtsuki,Beliakova-Habegger}. Since both proofs of this \lq key result\rq\, make use of surgery formulas from \cite{Lescop}, it seems necessary to clarify here how the present result is independent from \cite[Prop.~1.7.8]{Lescop}. 
\\
The original proof of \cite{Le-Murakami-Murakami-Ohtsuki} indeed makes use  of Lescop's third surgery formula.
But the alternative proof of \cite{Beliakova-Habegger} uses Lescop's \emph{first} surgery formula, \cite[1.4.8]{Lescop}, in terms of the so-called $\zeta$-coefficients extracted from the multivariable Alexander polynomial. In fact, \cite[Prop.~1.7.8]{Lescop} is proved in \cite[\S~6.4]{Lescop} using this first formula.\\
More importantly, both proofs of the \lq key result\rq\, use only a very special case of Lescop's formulas: in both \cite{Le-Murakami-Murakami-Ohtsuki} and \cite{Beliakova-Habegger}, the so-called \lq diagonalization lemma\rq\, (see e.g. \cite[Lem.~6]{Le-Murakami-Murakami-Ohtsuki}) is used to reduce the proof to the case of \emph{algebraically split} surgery presentations, that is, for links with zero linking numbers. This means that the \lq key result\rq\ uses a significantly weaker and simplified version of Lescop's formulas.\footnote{For algebraically split links, our invariant $\mu_k$ is just given by $-c_{k+1}$; this is to be compared with the much more involved formulas in Theorem \ref{cor:mun} and Definition \ref{U_n} for the general case.}
\end{remark}

We stress that the core of our first main result is the computation of the leading term of the LMO invariant, by purely diagrammatic methods. 
As a matter of fact, we expect that the techniques developed in this paper could be used in the future to address the next degree term of LMO, which is yet to be understood in general. 
\medskip

In order to prove Theorem \ref{thm1}, we provide in this paper a number of formulas identifying certain combinations 
of coefficients of the Kontsevich integral in terms of classical link invariants; see Sections 3 and 4. 
Such formulas are interesting in themselves, and rather few similar results are known up to now, see \emph{e.g.} \cite{HM,Stanford,Okamoto1997,Okamoto1998,JB}. 
Our formulas are derived from general factorization results, 
which show how certain local configurations in sums of coefficients in the Kontsevich integral, yield a factorization by simple link invariants; see Section \ref{sec:facto}.\\
As an example, the second main result of this paper 
uses these techniques to give an explicit identification for the $z^{n+1}$-coefficient $c_{n+1}$ of the Conway polynomial of an $n$-component link. 
This identification relies on the definition, outlined below, of a family of chord diagrams  
which are recursively built from a couple of low degree diagrams by simple local operations. 
Specifically, consider the following two local operations on chord diagrams, called \emph{inflation} and \emph{infection}: 
\begin{center}
\includegraphics[scale=0.8]{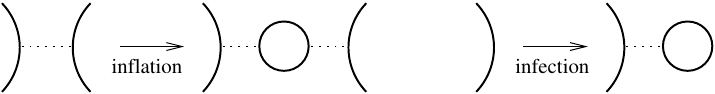}
\end{center}

For any integer $n\ge 1$, denote by $\mathcal{E}^-(n)$ the set of all (connected) chord diagrams on $n$ circles which are obtained from the two chord diagrams $\dessin{0.65cm}{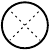}$ and $\dessin{0.65cm}{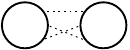}$ by iterated inflations, in all possible ways. 
Denote also by $\mathcal{P}(n)$ the set of all diagrams obtained from an element of $\mathcal{E}^-(n-1)$ by a single infection. 
Our second main result reads as follows.
\begin{thm}[Thm.~\ref{c_n}] \label{thm2}
Let $n\ge 2$. For any framed oriented $n$-component link $L$, we have 
 $$c_{n+1}(L) = \sum_{D \in \mathcal{E}^-(n) \cup \mathcal{P}(n)} C_L[D],$$
where $C_L[D]$ denotes the coefficient of $D$ in the (framed) Kontsevich integral of $L$.
\end{thm} 

Let us describe the simplest case $n=2$ more precisely. 
We have $\mathcal{E}^-(2)=\{\dessin{0.5cm}{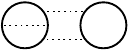};\dessin{0.5cm}{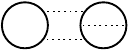};\dessin{0.5cm}{D32_2}\}$ and $\mathcal{P}(2)=\{\dessin{0.5cm}{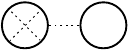};\dessin{0.5cm}{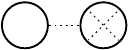}\}$.\footnote{We use the graphical convention that the circles are ordered from left to right.}
If $L$ is a framed oriented $2$-component link, then Theorem \ref{thm2} says that  $c_3(L)$ is given by 
$$C_L\Big[ \dessin{0.5cm}{D32_4} \Big] + C_L\Big[ \dessin{0.5cm}{D32_4b} \Big] + C_L\Big[ \dessin{0.5cm}{D32_2} \Big]+ C_L\Big[ \dessin{0.5cm}{D32_7} \Big]+ C_L\Big[ \dessin{0.5cm}{D32_7b} \Big]. $$

Figure \ref{fig:ex} gives typical examples of chord diagrams that are involved in the statement for higher values of $n$. 
\begin{figure}[!h]
 \includegraphics[scale=1]{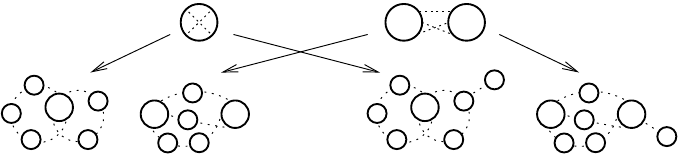}
 \caption{Two examples of elements in $\mathcal{E}^-(6)$ (left) and $\mathcal{P}(7)$ (right).}\label{fig:ex}
\end{figure}

The case $n=1$ is somewhat particular, as it involves a correction term. 
We have $\mathcal{E}^-(1)=\{\dessin{0.5cm}{D21_2}\}$ and $\mathcal{P}(1)=\emptyset$, and for a knot $K$ we have  
$$ c_2(K) = C_K\Big[\dessin{0.5cm}{D21_2}\Big] + \frac{1}{24},$$
a formula which is well-known to the experts (see Proposition \ref{c_2}). 

We stress that Theorem \ref{thm2} is of course related to the weight system of the Conway polynomial, 
computed in \cite{MMR} for solving the Melvin-Morton-Rozansky conjecture. 
Our statement and proof are however completely independent from \cite{MMR}. 
\medskip 

The paper is organized as follows. 
In Section 2, we review the various invariants of links and $3$-manifolds alluded to in the title of the paper. 
In Section 3, we identify certain combinations of coefficients in the framed Kontsevich integral in terms of classical invariants;  in particular, our factorization results are given in Section 3.2, while Theorem 2 is proved in Section 3.4. 
Section 4 is devoted to the invariants $\mu_n$; an explicit formula in terms of Conway coefficients and the linking matrix is given in Section 4.2. 
Finally, we prove Theorem 1 in Section 5. 

\begin{acknowledgments}
The authors would like to thank Christine Lescop for discussions regarding the relationship between Theorem \ref{thm1} and \cite{Lescop}. 
\end{acknowledgments}

\section{Preliminaries}

In this section we recall the necessary material for this paper. 
We start by a set a conventions that will be used throughout.

\subsection{Conventions and Notation}\label{sec:conv}

All $3$ manifolds will be assumed to be closed, compact, connected and oriented. 
All links live in the $3$-sphere $S^3$, and are assumed to be framed, oriented and ordered. 
\medskip 

Let $L=K_1\cup \cdots \cup K_n$ be an $n$-component link. 
Given a subset $I$ of $\{ 1, \ldots, n \}$, we set
$$ L_I := \bigcup_{i\in I} K_i\quad \textrm{ and }\quad  L_{\check{I}}:= L\setminus L_I.$$
\noindent We abbreviate $L_{\check{i}}=L_{\check{\{ i \}}}$. 

We denote by $l_{i,j}$ the linking number of the $i$th and $j$th component, and we denote by $fr_i=fr(K_i)$ the framing of the $i$th component. 

The linking matrix $\mathbb{L} \in M_n(\mathbb{Z})$ of $L$ is given by $\mathbb{L}_{i, i} = fr(K_i)$ and $\mathbb{L}_{i, j} = l_{i, j}$ if $i\neq j$. 
We denote by $\sigma_+(L)$, resp. $\sigma_-(L)$, the number of positive, resp. negative, eigenvalues of $\mathbb{L}$, 
so that its signature is given by $\sigma(L) = \sigma_+(L) - \sigma_-(L)$. 

\subsection{Conway polynomial and the $U_n$ invariant}

The \emph{Conway polynomial} is a renormalization of the Alexander polynomial, introduced by J.~Conway in the late $60$s. 
This is an invariant of (unframed) oriented links, which is a polynomial $\nabla$ in the variable $z$, defined by setting 
 $\nabla_U(z)=1$, 
where $U$ denotes the unknot, and 
 $$ \nabla_{L_+}(z) - \nabla_{L_-}(z) = z\nabla_{L_0}(z), $$
where $L_+$, $L_-$ and $L_0$ are three links that are identical except in a $3$-ball where they look as follows: 
\begin{center}
\begin{tikzpicture}
	\draw[->] [very thick] (0,0) -- (1,1);
	\draw[<-] [very thick] (0,1) -- (0.4,0.6);
	\draw (0.6,0.4) [very thick] -- (1,0);
	\draw (0.5,0) node[below]{$L_+$};
	
	\draw [very thick] (2,0) -- (2.4, 0.4);
	\draw[->] [very thick] (2.6,0.6) -- (3,1);
	\draw[<-] [very thick] (2,1) -- (3,0);
	\draw (2.5,0) node[below]{$L_-$};
	
	\draw[<-] [very thick] (4,1) to [bend left] (4,0);
	\draw[<-] [very thick] (4.5,1) to [bend right] (4.5,0);
	\draw (4.25,0) node[below]{$L_0$};
\end{tikzpicture}
\end{center}
We say that the three oriented links $(L_+,L_-,L_0)$ form a \emph{skein triple}, and a formula of the type above is typically called a skein formula.  

Denote by $c_k$ the coefficient of $z^k$ in the Conway polynomial. This is a $\mathbb{Z}$-valued link invariant, which satisfies the skein formula 
$c_{k+1}(L_+) - c_{k+1}(L_-) = c_k(L_0)$. 

For a knot $K$ we have $c_0(K)=1$, and for a $2$-component link $L=K_1\cup K_2$ we have $c_1(L)=l_{1,2}$. In general, 
the Conway polynomial of an $n$-component link $L$ has the form 
  $$\nabla_L(z)= \sum_{k=0}^N c_{n+2k-1}(L) z^{n+2k-1}. $$
\medskip

We can define the following link invariant from the Conway coefficients $c_k$.
\begin{definition}\label{def:U_n}
Let $n\ge 2$ be an integer.
We define an invariant $U_{n}$ of oriented $(n-1)$-component links by setting 
$$U_2(K) = c_2(K) - \frac{1}{24} $$
for a knot $K$, and the recursive formula  
$$U_{n+1}(L) = c_{n+1}(L) - \sum_{i = 1}^{n} U_{n}(L_{\check{i}}) \sum_{j \neq i} l_{i, j} $$
for an $n$-component link $L$ ($n\ge 3$). 
\end{definition}
For example, 
$$ U_3(K_1\cup K_2) = c_3(K_1\cup K_2) - l_{1,2}\left(c_2(K_1) + c_2(K_2) - \frac{1}{12}\right).  $$ 
\subsection{The Casson-Walker-Lescop invariant} \label{sec:casson}
The following is due to A.~Casson. 
\begin{theorem}[Casson] \label{casson}
There exists a unique $\mathbb{Z}$-valued invariant of integral homology spheres 
such that 
\begin{enumerate}
\item[(i).] $\lambda(S^3)=0$.  
\item[(ii).]  For any integral homology sphere $M$, any knot $K$ in $M$, and any $n\in \mathbb{Z}$, if $M_{K_n}$ denotes the result of $\frac{1}{n}$-Dehn surgery on $M$ along $K$, then:
 $$ \lambda(M_{K_{n+1}})-\lambda(M_{K_n})=c_2(K). $$
\end{enumerate}
Moreover, 
\begin{enumerate}
\item[(iii).] $\lambda$ changes sign under orientation reversal, and is additive under connected sum.
\item[(iv).]  The mod 2 reduction of $\lambda$ coincides with the Rochlin invariant.
\end{enumerate}
\end{theorem}

This is the \emph{Casson invariant} of integral homology spheres.  
Its existence was established by A.~Casson, who defined it in terms of count of conjugacy classes of irreducible $SU(2)$--representations of $\pi_1(M)$.  

In \cite{Walker1992}, K.~Walker extended the Casson invariant to a $\mathbb{Q}$-valued invariant of rational homology spheres $\lambda_W$, 
via a surgery formula.  
C.~Lescop then widely generalized the Casson-Walker invariant to all closed $3$-manifolds, 
by establishing  a global surgery formula involving the multivariable Alexander polynomial \cite{Lescop}. 
We denote by $\lambda_L$ this \emph{Casson-Walker-Lescop invariant}. Our convention is that, for a rational homology sphere $M$, we have 
$\lambda_L(M) = \frac{1}{2}\vert H_1(M)\vert \lambda_W(M)$.

\subsection{Universal invariants}

We now review the Kontsevich and LMO invariants, providing only the ingredients that are necessary for our purpose. 

\subsubsection{Chord diagrams and Jacobi diagrams}

Let us begin with introducing the spaces of diagrams in which the Kontsevich integral and LMO invariant take their values. 
We stress that our terminologies are somewhat different from the usual conventions of the literature: this is clarified in Remark \ref{rem:conv}. 

\begin{definition}
Let $X$ be some oriented $1$-manifold.
A \emph{chord diagram} $D$ on $X$ is a collection of copies of the unit interval, such that the set of all endpoints is embedded into $X$.
We call \emph{chord} any of these copies of the interval, and we call \emph{leg} any endpoint of a chord in $D$; 
the $1$-manifold $X$ is called the \emph{skeleton} of $D$.\\
A chord is called \emph{mixed}, resp. \emph{internal}, if its two legs lie on distinct, resp. the same, component(s) of the skeleton. \\
The  \emph{degree} of $D$ is defined as $deg(D) = |\{ \text{chords of\ } D \}| = \frac{1}{2} |\{ \text{legs of\ } D \}|$.
\end{definition}

\begin{definition}
We denote by $\mathcal{A}(X)$ the $\mathbb{Q}$-vector space generated by all chord diagrams on $X$, modulo the \emph{4T relation}: 
 $$ \dessin{1cm}{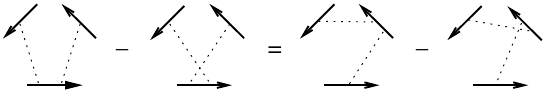}. $$
\end{definition}
\begin{remark}\label{rem:central}
As a consequence of the $4T$ relation, an \emph{isolated chord} (i.e.  a chord whose endpoints are met consecutively on the skeleton) 
commutes with any other chord, in the sense that we have the following relation in $\mathcal{A}(X)$: 
 $$ \dessin{0.9cm}{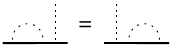}. $$
\end{remark} 

In what follows, we will almost exclusively be interested in \emph{chord diagrams on $n$ circles}, i.e. in the case where $X=\circlearrowleft^n$
consists of $n$ ordered, oriented copies of $S^1$. 
\medskip

\begin{definition}
A \emph{Jacobi diagram} is a trivalent graph whose trivalent vertices are equipped with a cyclic order on the incident edges.
The \emph{degree} of a Jacobi diagram is half its number of vertices. 
\end{definition}

\begin{definition}
We denote by $\mathcal{A}(\emptyset)$ the $\mathbb{Q}$-vector space generated by all Jacobi diagrams, modulo the \emph{AS and IHX relations}: 
 $$ \dessin{1.3cm}{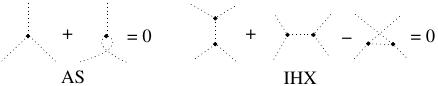}. $$
\end{definition}

\begin{notation}
For an element $x\in \mathcal{A}(\emptyset)$, and an integer $k\ge 0$, we denote by $(x)_{k}$, resp. $(x)_{\le k}$, its projection to the degree $k$ part $\mathcal{A}_{k}(\emptyset)$, resp. the degree $\le k$ part $\mathcal{A}_{\le k}(\emptyset)$.  
\end{notation}

\begin{remark}\label{rem:conv}
In the literature, the term \lq Jacobi diagram \rq\, more generally refers to unitrivalent diagrams whose univalent vertices lie disjointly on a (possible empty) $1$-manifold, 
subject to AS, IHX and an extra STU relation. 
Hence what we call \lq Jacobi diagrams\rq\, here are what experts know as \lq Jacobi diagrams on the empty set\rq, 
or \lq purely trivalent Jacobi diagrams\rq\,  -- this justifies our notation $\mathcal{A}(\emptyset)$. 
\end{remark}

We make use of the usual drawing conventions for chord and Jacobi diagrams: 
bold lines represent skeleton components while chords and graphs are drawn with dashed lines, 
and trivalent vertices are equipped with the counterclockwise ordering. Also, we assume when drawing elements of $\mathcal{A}(\circlearrowleft^n)$, 
that the circles are oriented counterclockwise and are ordered from left to right, unless otherwise specified.

\subsubsection{The Kontsevich integral}\label{sec:Kontsevich}

Let us give a quick overview of the Kontsevitch integral. 
We do not follow here Kontsevich's original definition \cite{Kontsevich1993}, but rather the combinatorial definition later provided in \cite{LM}.
Moreover, we will only give explicitly the low degree terms in the definitions, since these are all we need for the purpose of this paper. 
We refer the reader to \cite[\S 6.4]{Ohtsuki} for a detailed review.

Recall that a q-tangle is an oriented tangle, equipped with a consistent collection of parentheses on each of its linearly ordered sets of boundary points. 
A q-tangle can be (non-uniquely) decomposed into copies of the following  elementary q-tangles $I$, $X_{\pm}$, $C_{\pm}$ and $\Lambda_{\pm}$ (and those obtained by orientation-reversal on any component):\\[-0.1cm]
\begin{center}
\includegraphics[scale=0.8]{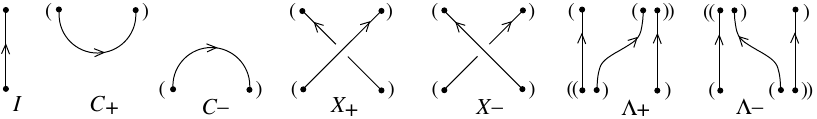}
\end{center}

The \emph{(framed) Kontsevich integral} can be determined by specifying its values on these elementary q-tangles.  

We set $Z(I)=\uparrow$, the diagram without chord of $\mathcal{A}(\uparrow)$, and 
$Z(C_\pm)=\sqrt{\nu}$,
where $\nu\in \mathcal{A}(\circlearrowleft)$ is the Kontsevich integral of the $0$-framed unknot $U_0$, computed in \cite{Bar-Natan-Garoufalidis-Rozansky-Thurston}:
\begin{equation}\label{eq:nu}
\begin{tikzpicture}
	\draw (0.5,0.5) node{$\nu =$};
	\draw [very thick] (1,0.5) arc(180:540:0.5);
	\draw[<-] [very thick] (1.5,1) -- (1.51,1);
	\draw (2.5,0.5) node{$+ \frac{1}{24}$};
	\draw [very thick] (3,0.5) arc(180:540:0.5);
	\draw[<-] [very thick] (3.5,1) -- (3.51,1);
	\draw [densely dashed] (3.1,0.3) -- (3.9,0.3);
	\draw [densely dashed] (3.1,0.7) -- (3.9,0.7);
	\draw (4.5,0.5) node{$- \frac{1}{24}$};
	\draw [very thick] (5,0.5) arc(180:540:0.5);
	\draw[<-] [very thick] (5.5,1) -- (5.51,1);
	\draw [densely dashed] (5.1,0.7) -- (5.9,0.3);
	\draw [densely dashed] (5.1,0.3) -- (5.9,0.7);
	\draw (7.2,0.5) node{$+$ degree $> 2$.};
\end{tikzpicture}
\end{equation}
Next we set
 $Z(X_\pm)=\textrm{exp}\big( \frac{\pm 1}{2} \dessin{0.7cm}{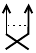} \big) = \sum_{k\ge 0} \frac{(\pm 1)^k}{2^k k!} \dessin{0.7cm}{X1}^{k}$, 
where the $k^{th}$ power denotes $k$ parallel dashed chords:
\begin{equation}\label{Z2}
 Z(X_\pm)=\dessin{0.85cm}{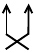} \pm\frac{1}{2} \dessin{0.85cm}{X1} + \frac{1}{8} \dessin{0.85cm}{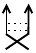} \pm \frac{1}{48} \dessin{0.85cm}{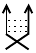} + \textrm{degree $>3$}. 
\end{equation}
Finally, set $Z(\Lambda_\pm)=\Phi^{\pm 1}$, where $\Phi\in \mathcal{A}( \uparrow\uparrow\uparrow )$ is the choice of a \emph{Drinfeld associator} 
(see e.g. \cite[App. D]{Ohtsuki}). At low degree, this gives 
 $$ Z(\Lambda_\pm)= \dessin{0.85cm}{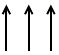} \pm\frac{1}{24} \left( \dessin{0.85cm}{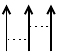} - \dessin{0.85cm}{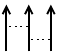}\right) + \textrm{degree $>3$}. $$

\begin{ex} \label{Calculs Kontsevich}
The following are well-known low degree computations for $Z(U_\pm)$, where $U_\pm$ denotes the $(\pm 1)$-framed unknot. 
\begin{center}
\begin{tikzpicture}
	\draw (0,0.5) node{$\hat{Z}(U_\pm) =$};
	\draw [very thick] (1,0.5) arc(180:540:0.5);
	\draw[<-] [very thick] (1.5,1) -- (1.51,1);
	\draw (2.5,0.5) node{$\pm \frac{1}{2}$};
	\draw [very thick] (3,0.5) arc(180:540:0.5);
	\draw[<-] [very thick] (3.5,1) -- (3.51,1);
	\draw [densely dashed] (3,0.5) -- (4,0.5);
	\draw (4.5,0.5) node{$+ \frac{1}{6}$};
	\draw [very thick] (5,0.5) arc(180:540:0.5);
	\draw[<-] [very thick] (5.5,1) -- (5.51,1);
	\draw [densely dashed] (5.1,0.7) -- (5.9,0.7);
	\draw [densely dashed] (5.1,0.3) -- (5.9,0.3);
	\draw (6.5,0.5) node{$- \frac{1}{24}$};
	\draw [very thick] (7,0.5) arc (180:540:0.5);
	\draw[<-] [very thick] (7.5,1) -- (7.51,1);
	\draw [densely dashed] (7.1,0.7) -- (7.9,0.3);
	\draw [densely dashed] (7.1,0.3) -- (7.9,0.7);
	\draw (9,0.5) node{  $\,\,\,\,\,+$ degree $> 2$.};
\end{tikzpicture}
\end{center}
\end{ex}

\subsubsection{The degree $\le 1$ part of the LMO invariant}\label{sec:LMO}

We now review the LMO invariant of closed oriented $3$-manifolds. 
Starting with an integral surgery presentation, this invariant is extracted from a renormalization $\check{Z}$ of the Kontsevich integral of this link via a family of sophisticated diagrammatic operations $\iota_n$. 
For the purpose of this paper, however, we only need the degree $\le 1$ part of the LMO invariant, 
and in particular we only need (a somewhat simplified definition of) the map $\iota_1$. 
We refer the reader to \cite{Le-Murakami-Murakami-Ohtsuki,Ohtsuki} for a complete definition.

\begin{definition}
Let $L$ be a framed oriented $n$-component link. 
We set $\check{Z}(L) := \hat{Z}(L) \# \nu^{\otimes n}$. In other words, in $\check{Z}$ we add a copy of $\nu$ to each circle component in $\hat{Z}(L)\in \mathcal{A}(\circlearrowleft^n)$.
\end{definition}

\begin{definition}
Given  a chord diagram $D$ on $n$ circles, we associate an element of $\mathcal{A}(\emptyset)$ as follows.  
For each circle component $c$ of $D$, if the number of legs on $c$ is $k$,  
\begin{itemize} 
 \item if $2\le k\le 4$, then replace $c$ by the portion of Jacobi diagram $T_k$, where 
  \[ T_2 = \dessin{0.05cm}{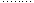},\quad T_3=\frac{1}{2}\dessin{0.9cm}{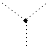},\quad T_4 = \frac{1}{6}\dessin{0.9cm}{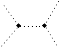} + \frac{1}{6} \dessin{0.9cm}{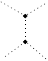}. \]
 \item if $k\le 1$ or $k\ge 5$, then map $D$ to $0$. 
\end{itemize}
Next, replace each copy of $\dessin{0.5cm}{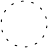}$ resulting from these replacements by a coefficient $(-2)$. 
The result is the desired element of $\mathcal{A}(\emptyset) $, which we denote by $\iota_1(D)$. 

By linearity this defines a map
\[ \iota_1: \mathcal{A}(\circlearrowleft^n)\longrightarrow \mathcal{A}(\emptyset).  \]
\end{definition}
\medskip 

Now, let $M$ be a closed $3$--manifold, and let $L$ be a framed $n$-component link in $S^3$ such that $M$ is obtained by surgery along $L$.  
Fix an orientation for the link $L$.
\begin{definition}
The degree $\le 1$ part of the \emph{LMO invariant} of $M$ is defined by 
 $$ Z^{LMO}_1(M):=  \left( \frac{\iota_1(\check{Z}(L))}{(\iota_1(\check{Z}(U_{+}))^{\sigma_+(L)} (\iota_1(\check{Z}(U_{-}))^{\sigma_-(L)} }\right)_{\le 1}\in \mathcal{A}_{\le 1}(\emptyset). $$
\end{definition}
This is an invariant of the $3$-manifold $M$: it does not depend on the choice of orientation of $L$, and does not change under Kirby moves.  

The denominator in the above formula is easily computed, see \cite{Ohtsuki}: 
\begin{equation}\label{Denominateur LMO}
 \left( \iota_1(\check{Z}(U_+))^{- \sigma_+(L)} \iota_1(\check{Z}(U_-))^{- \sigma_-(L)} \right)_{\le 1}= (-1)^{\sigma_+(L)} +
  \frac{(-1)^{\sigma_+(L)} \sigma(L)}{16} \dessin{0.8cm}{T}. 
\end{equation}
Moreover, the degree $0$ and $1$ parts of $Z_1^{LMO}(M)$ are clearly identified.
\begin{theorem}[\cite{Le-Murakami-Murakami-Ohtsuki,Beliakova-Habegger}]\label{LMOCasson}
Let $M$ be a closed $3$ manifold. We have 
$$Z_1^{LMO}(M)=l_0+l_1\dessin{0.8cm}{T}$$
where\\[-0.5cm] 
\begin{itemize}
	\item $l_0=\left\{
	\begin{array}{cl}
	 |H_1(M)|& \textrm{ if $M$ is a rational homology sphere,}\\
	 0 & \textrm{ otherwise.}
	\end{array}
	\right.$
	\item $l_1 = \frac{(-1)^{\beta_1(M)}}{2} \lambda_L(M)$, where $\beta_1$ denotes the first Betti number.
\end{itemize}
\end{theorem}
The second point of Theorem \ref{LMOCasson} is the key result in establishing our surgery formula for the Casson-Walker-Lescop invariant.  

\begin{remark}\label{rem:iota}
Our definition of the $\iota_1$ map differs from the usual one in that we map to zero all diagrams of degree $\ge 5$. This modification is harmless since, with the original definition of $\iota_1$, such diagrams cannot contribute to the degree $\le 1$ part of the LMO invariant.
\end{remark}

\section{Coefficients of the Kontsevich integral}

In this section we identify certain combinations of coefficients in the framed Kontsevich integral in terms of classical invariants. 

\subsection{Operations on Jacobi diagrams}

\subsubsection{Preliminaries}

We begin by introducing some notations and tools that will be used throughout the rest of the paper. 

\begin{notation} \label{nota:coeff}
Let $S$ be an element of $\mathcal{A}(\circlearrowleft^n)$. Let $D$ be a chord diagram on $n$ circles.
We denote by $C[D](S)$ the coefficient of $D$ in $S$.
In particular, we set 
$$C_L[D] := C[D](\hat{Z}(L)),$$
for a framed oriented link $L$, and we denote by $C[D]$ the assignment $L\mapsto C_L[D]$. 
\end{notation}
Of course, these quantities are in general not well-defined, since elements of $\mathcal{A}(\circlearrowleft^n)$ are subject to the $4T$ relation. 
However, taking appropriate combinations of such coefficients may yield well-defined link invariant: this is recalled in the first of the next three rather simple and well-known lemmas, whose proofs are omitted (proofs can be found in \cite{Casejuane}). 

\begin{lemma}[Invariance] \label{Critere Invariance}
Let $D_1, \ldots, D_k$ be chord diagrams on $n$ circles. 
Then  $X := C[D_1] + \ldots + C[D_k]$ defines an invariant of framed oriented $n$-component links if and only if $X$ vanishes on any linear combination of chord diagrams arising from a $4T$ relation. 
\end{lemma}

\begin{lemma}[Disjoint Union] \label{Factorisation}
Let $D$ be a chord diagram on $n$ circles that splits into a disjoint union $D=D_I\sqcup D_J$.
Then for any framed oriented $n$-component link $L$ we have 
$$C_L[D] = C_{L_I}[D_I] \times C_{L_J}[D_J], $$
where $L_I$ and $L_J$ are sublinks of $L$ corresponding to the components of $D_I$ and $D_J$. 
\end{lemma}

\begin{lemma}[Skein] \label{Obvious lemma}
Let $D$ be a chord diagram of degree at most $4$.
Let $L_+$ and $L_-$ be the first two terms of a skein triple at a crossing $c$ between the $i$th and $j$th components (possibly $i=j$).
Then $C_{L_+}[D] - C_{L_-}[D]$ is given by 
$$ 
\left\{ \begin{array}{ll}
C[D] \Big(\dessin{0.9cm}{X1}\Big) & \text{ if $D$ has $\le 2$ chords between components $i$ and $j$,}\\
C[D] \Big(\dessin{0.9cm}{X1} + \frac{1}{24}\dessin{0.9cm}{X3}\Big) & \text{ if $D$ has $\ge 3$ chords between components $i$ and $j$,}
\end{array}\right.
$$
where we only show the local contribution to $\hat{Z}(L)$ given by the crossing $c$. 
\end{lemma}

\subsubsection{Inflations and Infections}

We now introduce several local operations on chord diagrams. The first one is a standard one: 

\begin{definition}
Let $D$ be a chord diagram on $n$ circles, with at least one chord, as shown on the left-hand side of the figure below. 
A \emph{smoothing} of $D$ along this chord is a chord diagram $D_0$ obtained from the following operation:
\begin{center}
\begin{tikzpicture}
	\draw[<-] [very thick] (0,1) -- (1,0);
	\draw[->] [very thick] (0,0) -- (1,1);
	\draw (-0.2,0.5) node{$D$};
	\draw [densely dashed] (0.25,0.75) -- (0.75,0.75);
	\draw[->] (1.25,0.5) -- (2.25,0.5);
	\draw[<-] [very thick] (2.5,1) to[bend left] (2.5,0);
	\draw[<-] [very thick] (3,1) to[bend right] (3,0);
	\draw (3.4,0.5) node{$D_0$};
\end{tikzpicture}
\end{center}
\end{definition}
\noindent
Figure \ref{fig:lissage} gives two examples of smoothings (along the chord marked with a $\ast$). 
\begin{figure}[!h]
\begin{center}
 \includegraphics[scale=0.8]{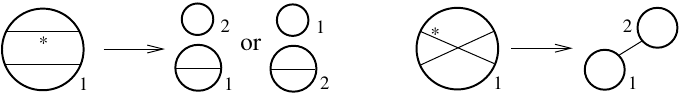} 
 \caption{Smoothing chord diagrams on a circle.}\label{fig:lissage}
\end{center}
\end{figure}

\noindent If the chord lies on two disjoint components $i$ and $j$ ($i<j$), then these two circles become a single component of $D_0$, labeled by $i$, and the circles $k>j$ are re-labeled by $(k-1)$.  
Otherwise, as Figure \ref{fig:lissage} illustrates, the skeleton of $D_0$ has $(n+1)$ components, and the $(n+1)$th component is one of the two circles arising from the smoothing.   

The next two operations, called inflation and infection, will provide recursive tools for building chord diagrams with useful properties, in any degree. 
\begin{definition}
Let $D$ be a chord diagram on $n$ circles, and let $c$ be a chord of $D$.
The \emph{inflation of $D$ along $c$} is the following local operation: \\[-0.2cm] 
\begin{center}
\begin{tikzpicture}
	\draw [very thick] (0,0) arc(270:450:0.5);
	\draw[<-] [very thick] (0,1) -- (0.01,1);
	\draw (0,0) node[left]{$i$};
	\draw [densely dashed] (0.5,0.5) -- (1.5,0.5);
	\draw (1,0.5) node[below]{$c$};
	\draw [very thick] (2,0) arc(270:90:0.5);
	\draw[->] [very thick] (1.99,0) -- (2,0);
	\draw (2,0) node[right]{$j$};
	\draw[->] (2.25,0.5) -- (3.25,0.5);
	\draw [very thick] (3.5,0) arc(270:450:0.5);
	\draw[<-] [very thick] (3.5,1) -- (3.51,1);
	\draw (3.5,0) node[left]{$i$};
	\draw [densely dashed] (4,0.5) -- (4.5,0.5);
	\draw [very thick] (4.5,0.5) arc(180:540:0.5);
	\draw (5,0) node[below]{$n+1$};
	\draw[<-] [very thick] (5,1) -- (5.01,1);
	\draw [densely dashed] (5.5,0.5) -- (6,0.5);
	\draw [very thick] (6.5,0) arc(270:90:0.5);
	\draw[->] [very thick] (6.49,0) -- (6.5,0);
	\draw (6.5,0) node[right]{$j$};
\end{tikzpicture}
\end{center}
\end{definition}

\begin{definition}
Let $D$ be a chord diagram on $n$ circles, and let $I$ be an interval in the skeleton, whose interior is disjoint from all chords. 
The \emph{infection of $D$ along $I$} is the following local operation:\\[-0.2cm]  
\begin{center}
\begin{tikzpicture}
	\draw [very thick] (0,0) arc(270:450:0.5);
	\draw[<-] [very thick] (0,1) -- (0.01,1);
	\draw (0,0) node[left]{$i$};
	\draw (0.5,0.5) node[right]{$I$};
	\draw[->] (1.25,0.5) -- (2.25,0.5);
	\draw [very thick] (2.5,0) arc(270:450:0.5);
	\draw[<-] [very thick] (2.5,1) -- (2.51,1);
	\draw (2.5,0) node[left]{$i$};
	\draw [densely dashed] (3,0.5) -- (4,0.5);
	\draw [very thick] (4,0.5) arc(180:540:0.5);
	\draw[<-] [very thick] (4.5,1) -- (4.51,1);
	\draw (4.8,0) node[right]{$n+1$};
\end{tikzpicture}
\end{center}
We call \emph{infection along the $i$th component} of $D$ the result of an infection along any of the intervals bounded by the legs on the $i$th circle component.
\end{definition}

\begin{remark}\label{rem:lissage}
Smoothing the chord that appears in an infection on a chord diagram $D$ gives back $D$. 
Likewise, after some inflation on $D$, smoothing either of both chords attached to the $(n+1)$th circle, gives back $D$.
\end{remark}

\subsubsection{Essential diagrams}\label{sec:essentiels}

We now introduce several families of chord diagrams. 
We start with a general definition. 
\begin{definition}\label{def:close}
Let $D$ be a chord diagram, and let $D'$ be an element of $\mathcal{A}(\emptyset)$. 
We say that $D$ \emph{closes into $D'$} when $\left(\iota_1(D)\right)_{\le 1} = D'$.
\end{definition}

We first consider diagrams that close into a nonzero constant. 
By the definition of $\iota_1$, a \emph{connected} diagram closes into the empty diagram with nonzero coefficient if, 
and only if each circle component has exactly two legs; such diagrams will be called \lq chain of circles\rq\, in the rest of this paper:  
\begin{definition}\label{rem:chain}
A \emph{chain of $n$ circles} is a degree $n$ chord diagram obtained by $(n-1)$ successive inflations on the diagram $\dessin{0.6cm}{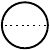}$, 
up to permutation of the circle labels. 
\end{definition}
For example, chains of $1$, $2$ and $3$ circles are of the form 
$\dessin{0.5cm}{D11}$, $\dessin{0.5cm}{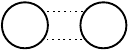}$ and $\dessin{0.6cm}{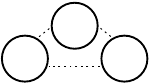}$, respectively. 
It is immediately verified that any chain of circles closes into $-2$. 
\medskip 

We now consider diagrams that close into the Theta-shaped diagram $\dessin{0.5cm}{T}$. 
\begin{definition}
A \emph{connected} chord diagram is called  
\begin{itemize}
\item a \emph{$\oplus$-essential diagram} if it closes into  $\dessin{0.5cm}{T}$ with positive coefficient,
\item a \emph{$\ominus$-essential diagram} if it closes into  $\dessin{0.5cm}{T}$ with negative coefficient, 
\item an \emph{essential diagram} if it either a $\oplus$-essential or $\ominus$-essential diagram.
\end{itemize}
We denote respectively by $\mathcal{E}^+(n)$ and $\mathcal{E}^-(n)$, the set of $\oplus$-essential and $\ominus$-essential diagrams on $n$ circles. 
We also set $\mathcal{E}(n):=\mathcal{E}^+(n)\cup \mathcal{E}^-(n)$.
\end{definition}

Before further investigating these families of diagrams, let us give low-degree examples. 
\begin{ex}\label{ex:essentiels}
All $\oplus$-essential diagrams on $\le 2$ circles are given by 
$$ \mathcal{E}^+(1) = \{  \dessin{0.5cm}{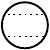} \} \quad \textrm{and} \quad \mathcal{E}^+(2) = \{ \dessin{0.5cm}{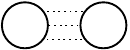} \, ; \, \dessin{0.5cm}{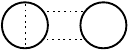}  \, ; \, \dessin{0.5cm}{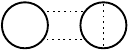} \},  $$
and they all close into $\frac{1}{6}\times \dessin{0.5cm}{T}$. 
\\
All $\ominus$-essential diagrams on $\le 2$ circles close into $-\frac{1}{3}\times \dessin{0.5cm}{T}$ and are given by 
$$ \mathcal{E}^-(1) = \{ \dessin{0.5cm}{D21_2} \} \quad \textrm{and} \quad \mathcal{E}^-(2) = \{ \dessin{0.5cm}{D32_2}\,;\,\dessin{0.5cm}{D32_4}  \, ; \, \dessin{0.5cm}{D32_4b}\}. $$
\end{ex}

More generally, the following combinatorial criterion can easily be deduced from the definition of the map $\iota_1$. 
\begin{lemma} \label{Conditions combinatoires}
Let $D$ be a chord diagram. Then $D$ is essential if, and only if it is of the one of the following two types: 
      \begin{itemize}
	\item  $D$ contains one circle with $4$ legs, and all other circles have $2$ legs, 
	\item  $D$ contains two circles with $3$ legs, and all other circles have $2$ legs. 
     \end{itemize}
It follows that an essential diagram on $n$ circles has always degree $n+1$. 
\end{lemma}

We now relate essential diagrams to the inflation operation. 
\begin{proposition} \label{Inflations et Essentiels}
Inflation on a $\oplus$-essential (resp. $\ominus$-essential) diagram of degree $n$ yields a $\oplus$-essential (resp. $\ominus$-essential) diagram of degree $n+1$, for all $n\ge 1$.\\
Conversely, for $n\ge 3$, any $\oplus$-essential (resp. $\ominus$-essential) diagram of degree $n+1$ is the inflation of a $\oplus$-essential (resp. $\ominus$-essential) diagram of degree $n$, up to permutation of the circle labels. 
\end{proposition}
\begin{proof}[Proof]
The first part of the statement is rather easily verified, as follows. Firstly, inflation preserves connectivity. 
Secondly,  if $D$ is obtained by inflation on (say) a $\oplus$-essential diagram $\tilde{D}$, then one can freely chose, when applying the map $\iota_1$ to $D$, to first act on the $(n+1)$th circle, which is replaced by an edge by inserting a copy of $T_2$: the result is the diagram $\tilde{D}$ (with coefficient $1$), which by definition closes into $\dessin{0.5cm}{T}$ with positive coefficient.\\
Conversely, since $n\ge 3$, the skeleton of an essential diagram $D$ of degree $(n+1)$ has at least $3$ circles. 
Lemma \ref{Conditions combinatoires} then tells us that $D$ has at least one circle component with exactly two legs. 
Since $D$ is connected, these two legs are the endpoints of two (distinct) mixed chords, which allows us to regard $D$ as the result of an inflation.
\end{proof}

\begin{remark}\label{rem:seeds}
By combining Proposition \ref{Inflations et Essentiels} with Example \ref{ex:essentiels}, we have that, up to permutation of the circle labels, 
any $\oplus$-essential diagram is obtained by iterated inflations from either $\dessin{0.65cm}{D21_1}$ or $\dessin{0.65cm}{D32_1}$, 
and that any $\ominus$-essential diagram is obtained by iterated inflations from either $\dessin{0.65cm}{D21_2}$ or $\dessin{0.65cm}{D32_2}$. 
\end{remark}

We close this section by a technical result on $\ominus$-essential diagrams. 
\begin{lemma}\label{unpeuutilequandmeme}
Let $n\ge 3$ be an integer. 
For any $\ominus$-essential diagram of degree $n$, smoothing a mixed chord always yields a $\ominus$-essential diagram of degree $(n-1)$. 
Conversely, any $\ominus$-essential diagram of degree $(n-1)$ can be obtained in this way. 
\end{lemma}
\begin{proof}[Proof]
Let $D$ be a $\ominus$-essential diagram of degree $n$. 
Remark \ref{rem:seeds} above tells us that $D$ is obtained by iterated inflations from either $\dessin{0.65cm}{D21_2}$ or $\dessin{0.65cm}{D32_2}$. 
As noted in Remark \ref{rem:lissage}, smoothing a (mixed) chord that appeared in one of these inflations yields the diagram before inflation, 
which is a $\ominus$-essential one. So it only remains to observe that smoothing a mixed chord of $\dessin{0.65cm}{D32_2}$ always yields $\dessin{0.65cm}{D21_2}$. 
\end{proof}
Note that the same result holds for $\oplus$-essential diagrams, but is not needed for this paper. 

\subsection{Factorization results}\label{sec:facto}

We now give a collection of factorization results for invariants that are defined as sums of coefficients of chord diagrams in the Kontsevich integral, containing certain particular chord configurations. 

\begin{proposition} \label{Factorisation fr}
Let $\mathcal{D}$ be a set of chord diagrams such that $X = \sum_{D \in \mathcal{D}} C[D]$ is a link invariant. 
Suppose that, for some index $i$, none of the diagrams in $\mathcal{D}$ contains an internal chord on the $i$th circle. 
Let $\mathcal{D}_i$ be the collection of diagrams obtained from those in $\mathcal{D}$ by adding an internal chord on the $i$th circle, in all possible ways. 
Then, for any framed oriented link $L$, we have  
$$\sum_{D' \in \mathcal{D}_i} C_L[D'] = \frac{1}{2} fr_i \times X(L). $$
\end{proposition}
\begin{proof}[Proof]
Set $Y_i=\sum_{D' \in \mathcal{D}_i} C[D']$. 

We first verify that $Y_i$ indeed is a link invariant, using the Invariance Lemma \ref{Critere Invariance}.  
We develop the argument below, although this straightforward (but somewhat lengthy) step will often be ommited in the rest of this paper. 
Consider a $4T$ relation $R'$. 
It suffices to consider the case where $R'$ involves at least one diagram from $\mathcal{D}_i$.  
There are two possibilities.
\begin{enumerate}
	\item The internal chord on the $i$th circle is not involved in $R'$. Since at least one diagram involved in $R'$ has an internal chord on the $i$th circle, this is actually the case for all of them. 
        The four diagrams involved in $R'$ can then be regarded as obtained, by adding an internal chord on the $i$th circle in some way, from diagrams $D_1$, $D_2$, $D_3$ and $D_4$ 
        which satisfy a $4T$ relation $R$ of the form $D_1 - D_2 = D_3 - D_4$.  
        Since $X$ is a link invariant, it satisfies this relation: this proves that $Y_i$ satisfies $R'$. 
        Indeed a diagram involved in $4T$ is in $\mathcal{D}$ if and only if the corresponding diagram involved in $R'$ is in $\mathcal{D}_i$. 
	\item  The internal chord on the $i$th circle is involved in $R'$. 
	We can then write $R'$ as $D'_1 - D'_2 = D'_3 - D'_4$, where $D'_1$ and $D'_2$ both contain an internal chord on the $i$th circle.
	Hence $D'_1$ and $D'_2$ are in $\mathcal{D}_i$, and $Y_i$ vanishes on the left-hand term of relation $R'$. 
	For the remaining two diagrams, there are two cases.  
	If $D'_3$ also contains an internal chord on the $i$th circle, then so does $D'_4$, and both diagrams are in $\mathcal{D}_i$; otherwise, neither $D'_3$ nor $D'_4$ is in $\mathcal{D}_i$. 
	In any case $Y_i$ vanishes on the right-hand term of relation $R'$.  
\end{enumerate}
Thus $Y_i$ is an invariant, and it remains to show the factorization formula. 

Let $L_+ = K_1^+ \cup \ldots \cup K_n^+$ and $L_- = K_1^- \cup \ldots \cup K_n^-$ be the first two terms of a skein triple at an internal crossing of the $i$th component. 
Observe that $\hat{Z}(L_+)$ and $\hat{Z}(L_-)$ only differ by internal chords on the $i$th circle, so that $X(L_+)=X(L_-)$.
By the Skein Lemma \ref{Obvious lemma}, we have\footnote{As in the Skein  Lemma \ref{Obvious lemma}, we only show here the local contribution to $\hat{Z}(L)$ of the crossing involved in the skein triple; we will always implicitely do so in the rest of the paper. }
\begin{eqnarray*}
Y_i(L_+) - Y_i(L_-) & = & \sum_{D' \in \mathcal{D}_{i}} C[D'] \Big( \dessin{0.9cm}{X1} \Big) \\
                               & = & \sum_{D \in \mathcal{D}} C[D] \Big( \dessin{0.9cm}{X0} \Big) \\
                               & = & X(L_\pm).
\end{eqnarray*}
Here, the second equality follows directly from the definition of $\mathcal{D}_i$, while the third equality follows from the definition of $X$. 
But, since $X(L_+)=X(L_-)$ and $ fr(K_i^+)- fr(K_i^-) = 2$, we also have $\frac{1}{2} \times fr(K_i^+) \times X(L_+) - \frac{1}{2} \times fr(K_i^-) \times X(L_-)=X(L_\pm)$. 
Hence the two invariants in the statement satisfy the same skein formula. \\ 
Now, by successive internal crossing changes on the $i$th component, 
we can deform any link into a link $\tilde{L}$ whose $i$th component is isotopic to a copy of the unknot $U_0$, with no internal crossing, 
or a copy of $U_+$, with a single, isolated positive kink: it suffices to check that, in both cases, the formula of the statement holds. 
If the $i$th component of  $\tilde{L}$ is a copy of $U_0$, then $\hat{Z}(\tilde{L})$ contains no diagram with an internal chord on the $i$th circle, hence $Y_i$ vanishes, and the formula holds. 
If the $i$th component of  $\tilde{L}$ is a copy of $U_+$, then the isolated positive kink locally contributes to $\hat{Z}(\tilde{L})$ as recalled in (\ref{Z2}), 
and in particular gives on the $i$th circle:  
 $$ \dessin{0.6cm}{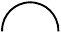} + \frac{1}{2}\dessin{0.6cm}{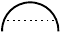} + \textrm{terms with $>1$ internal chords.} $$
By Remark \ref{rem:central}, there is only one diagram with isolated internal chord on the $i$th circle in $\mathcal{D}_i$, 
which shows that $Y_i(\tilde{L})$ equals $\frac{1}{2}X(\tilde{L})$ in this case, thus showing the desired formula. 
\end{proof}
An example of application of Theorem \ref{Factorisation fr} will be given in Lemma \ref{lem:D11}. 
\medskip 
\begin{proposition} \label{Factorisation lk}
Let $\mathcal{D}$ be a set of chord diagrams such that $X = \sum_{D \in \mathcal{D}} C[D]$ is a link invariant. 
Suppose that none of the diagrams in $\mathcal{D}$ contains a mixed chord between the $i$th and $j$th circles ($i\neq j$).
Let $\mathcal{D}_{ij}$ be the collection of diagrams obtained from those in $\mathcal{D}$ by adding a chord between the $i$th and $j$th circles, in all possible ways. 
Then, for any framed oriented link $L$, we have  
$$\sum_{D \in \mathcal{D}_{ij}} C_L[D] = l_{i, j} \times X(L).$$
\end{proposition}

\begin{proof}[Proof]
Set $Y_{ij}=\sum_{D \in \mathcal{D}_{ij}} C[D]$. 
The fact that $Y_{ij}$ is a link invariant is shown by similar arguments as in the proof of Proposition \ref{Factorisation fr}, namely by considering a $4T$ relation and analysing the various cases, depending on whether the diagrams involved in this relation involve a mixed chord between $i$ and $j$ or not.  \\
Now consider the first two terms $L_+$ and $L_-$ of a skein triple at a (mixed) crossing between the $i$th and $j$th components. 
Note that $\hat{Z}(L_+)$ and $\hat{Z}(L_-)$ only differ by terms containing chords  between the $i$th and $j$th circles, so that $X(L_+)=X(L_-)$. 
It follows from  the Skein Lemma \ref{Obvious lemma}, and the definitions of $\mathcal{D}_{ij}$ and $X$, that 
\begin{eqnarray*}
Y_{ij}(L_+) - Y_{ij}(L_-) & = & \sum_{D' \in \mathcal{D}_{ij}} C[D'] \Big( \dessin{0.9cm}{X1} \Big) \\
                                     & = & \sum_{D \in \mathcal{D}} C[D] \Big( \dessin{0.9cm}{X0} \Big) \\
                                     & = & X(L_\pm)                                
\end{eqnarray*}
which indeed coincides with  $l_{i,j}(L_+) \times X(L_+) - l_{i,j}(L_-) \times X(L_-)$ (since $l_{i,j}(L_+) - l_{i,j}(L_-) = 1$). 
It remains to observe that, by a sequence of crossing changes between the $i$th and $j$th components and isotopies, any link can be deformed into a link where the $i$th and $j$th components are geometrically split. The desired formula is easily checked for such links, and the result follows. 
\end{proof}
A simple application of Proposition \ref{Factorisation lk} is given in Lemma \ref{lem:D12}. 

More generally, the following is a consequence of Theorem \ref{Factorisation lk}, which identifies the invariant underlying an infection. 
\begin{proposition} \label{Inflation et invariance}
Let $\mathcal{D}$ be a set of chord diagrams such that $X = \sum_{D \in \mathcal{D}} C[D]$ is an $n$-component link invariant. 
Let $\mathcal{D}_I$  be the collection of diagrams obtained from those in $\mathcal{D}$ by all possible infections on the $i$th circle, for some $i\in\{1,\cdots,n\}$. 
Then, for any framed oriented $(n+1)$-component link $L$, we have  
$$\sum_{D' \in \mathcal{D}_I} C_L[D'] = l_{i, n+1}\times X(L_{\check{n+1}}).$$
\end{proposition}
\begin{proof}[Proof]
Denote by $\mathcal{D}_\circ$ the collection of diagrams obtained from those in $\mathcal{D}$ by adding a copy of $\dessin{0.5cm}{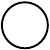}$ labeled by $(n+1)$. 
Then for any framed oriented $(n+1)$-component link $L$, we have  by the Disjoint Union Lemma \ref{Factorisation} that 
$\sum_{D_\circ \in \mathcal{D}_\circ} C_L[D_\circ] = X(L_{\check{n+1}})$. 
It then suffices to apply Proposition \ref{Factorisation lk} to the $i$th and $(n+1)$th circles of all diagrams in $\mathcal{D}_\circ$. 
\end{proof}

Similarly, the next result identifies the invariant underlying an inflation. 
\begin{proposition} \label{Factorisation gonflage}
Let $\mathcal{D}$ be a set of chord diagrams such that $X = \sum_{D \in \mathcal{D}} C[D]$ is an $n$-component link invariant. 
Let $\mathcal{D}_{G}$ be the collection of diagrams obtained from those in $\mathcal{D}$ by adding the following local diagram, called \emph{inflated chord}, 
in all possible ways\\[-0.3cm]
\begin{center}
\begin{tikzpicture}
	\draw (0,0.5) node[left]{$i$};
	\draw[-] [very thick] (0,0) -- (0,1);
	\draw [densely dashed] (0,0.5) -- (1,0.5);
	\draw [very thick] (1,0.5) arc(180:540:0.5);
	\draw[<-] [very thick] (1.5,1) -- (1.51,1);
	\draw (1.5,0) node[below]{$n+1$};
	\draw [densely dashed] (2,0.5) -- (3,0.5);
	\draw (3,0.5) node[right]{$j$};
	\draw[-] [very thick] (3,0) -- (3,1);
\end{tikzpicture}
\end{center}
Then, for any framed oriented $(n+1)$-component link $L$, we have 
\begin{eqnarray*} 
\sum_{D \in \mathcal{D}_{G}} C_L[D] = l_{i, n+1} \times l_{j, n+1} \times X(L_{\check{n+1}}), & \textrm{ if $i\neq j$,} \\
\sum_{D \in \mathcal{D}_{G}} C_L[D] = \frac{1}{2} l^2_{i, n+1} \times X(L_{\check{n+1}}), & \textrm{ if $i=j$.} 
\end{eqnarray*}
\end{proposition}
\begin{proof}[Proof] 
The case $i\neq j$ is a rather immediate consequence of the previous results. Indeed, in this case, the elements of $\mathcal{D}_{G}$ can be seen as obtained from those of $\mathcal{D}$ by first, all possible infections on the $i$th circle, then all possible ways of adding a mixed chord between the $j$th and $(n+1)$th circles (the latter one resulting from the infection). 
The result thus follows from Propositions \ref{Factorisation lk} and \ref{Inflation et invariance}.\\
We now prove the case $i=j$.
The fact that $Y_{G}:=\sum_{D \in \mathcal{D}_{G}} C[D]$ indeed defines an invariant is done in a similar way as in the previous proofs, and is left as an exercise to the reader.  
We prove that $Y_{G}(L)$ coincides with $\frac{1}{2} l^2_{i,n+1} \times X(L_{\check{n+1}})$ by showing that both invariants have same variation formula under a crossing change between the $i$th and $(n+1)$th components: since these invariants both vanish on links where these two components are geometrically split, the result will follow.\\
Let $(L_+, L_-, L_0)$ be a skein triple at a crossing between the $i$th and $(n+1)$th components. 
On one hand, from the definitions of $\mathcal{D}_{G}$, we have 
 \begin{eqnarray*}
Y_{G}(L_+) - Y_{G}(L_-) & = & \sum_{D' \in \mathcal{D}_{G}} C[D'] \Big( \dessin{0.9cm}{X1} \Big) \\
                                       & = & \sum_{\widetilde{D} \in \mathcal{D}_I} C[\widetilde{D}] \Big( \dessin{0.9cm}{X0} \Big) 
\end{eqnarray*}
where $\mathcal{D}_I$ denotes the set of all diagrams obtained from $\mathcal{D}$ by an infection on the $i$th circle, in all possible ways. 
The second equality thus holds by the fact that inserting an inflated chord on the $i$th circle is achieved by first, an infection on the $i$th circle, followed by the insertion of a mixed chord. 
Now, by definition of the Kontsevich integral at a negative crossing  (\ref{Z2}), for any $\widetilde{D} \in \mathcal{D}_I$, we have 
 $$ C_{L-}[\widetilde{D}] = C[\widetilde{D}] \Big( \dessin{0.9cm}{X0} \Big) -\frac{1}{2} C[\widetilde{D}] \Big( \dessin{0.9cm}{X1} \Big), $$
where the local picture still involves the $i$th and $(n+1)$th circle. Hence by subsitution we obtain 
 \begin{eqnarray*}
Y_{G}(L_+) - Y_{G}(L_-) & = & \sum_{\widetilde{D} \in \mathcal{D}_I} C_{L-}[\widetilde{D}] + \frac{1}{2} \sum_{\widetilde{D} \in \mathcal{D}_I}  C[\widetilde{D}] \Big( \dessin{0.9cm}{X1} \Big) \\
                                       & = & l_{i,n+1}(L_-)\times X((L_\pm)_{\check{n+1}}) + \frac{1}{2}X((L_\pm)_{\check{n+1}}).
\end{eqnarray*}
Here, the last equality uses the definition of $X$ and Proposition \ref{Factorisation lk}, and the fact that $(L_+)_{\check{n+1}}=(L_-)_{\check{n+1}}$.
On the other hand, using simply the fact that $ l_{i,n+1}(L_+) = l_{i,n+1}(L_-)+1$, 
we have
$$ \frac{1}{2} l^2_{i,n+1}(L_+) \times X((L_+)_{\check{n+1}}) - \frac{1}{2} l^2_{i,n+1}(L_-) \times X((L_-)_{\check{n+1}})  $$
$$ = \frac{1}{2} X((L_\pm)_{\check{n+1}})  \left( l^2_{i,n+1}(L_+)  - l^2_{i,n+1}(L_-) \right) $$
$$ = \frac{1}{2} X((L_\pm)_{\check{n+1}}) \left( 2l_{i,n+1}(L_-) + 1 \right),  $$
which shows that the two invariants in the statement satisfy the same skein formula. This concludes the proof. 
\end{proof}

\subsection{Some results in low degree}

In this section, we identify, in low degrees, some combinations of coefficients of the Kontsevich integral in terms of classical invariants. 
We begin with a few simple applications of our factorization results, most of which are well-known to the experts.  

The following is an elementary application of Theorem \ref{Factorisation fr} and the obvious formula $C_K\Big[ \dessin{0.5cm}{D01}\Big]=1$. 
\begin{lemma}\label{lem:D11}
Let $K$ be a framed oriented knot. We have 
$$ C_K\Big[ \dessin{0.5cm}{D11}\Big] = \frac{1}{2} fr(K).$$
\end{lemma}

The following can be seen as a consequence of either Theorem \ref{Factorisation lk} or \ref{Inflation et invariance}. 
\begin{lemma}\label{lem:D12}
Let $L$ be a framed oriented $2$-component link. We have 
$$ C_L\Big[ \dessin{0.5cm}{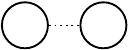}\Big] = l_{1,2}.$$
\end{lemma}

We next give two simple examples of applications of Proposition \ref{Factorisation gonflage}. 
The first example uses the case $i=j$ of the proposition.
\begin{lemma}\label{lem:D22}
Let $L$ be a framed oriented $2$-component link. We have 
$$ C_L\Big[ \dessin{0.5cm}{D22}\Big] =\frac{1}{2} l^2_{1,2}.$$
\end{lemma}
\noindent The second example uses the case $i\neq j$ of Proposition \ref{Factorisation gonflage}, combined with Lemma \ref{lem:D22} above. 
\begin{lemma}\label{Invariant particulier}
Let $L$ be a framed oriented $3$-component link. Then for $\{i,j,k\}=\{1,2,3\}$ we have 
$$ C_L\Big[ \dessin{0.9cm}{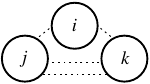}\Big] + C_L\Big[ \dessin{0.9cm}{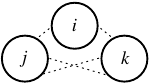}\Big] = \frac{1}{2} \times l_{i, j} \times l_{i, k} \times l_{j, k}^2. $$
\end{lemma}
\begin{remark}
Direct proofs of the above four lemmas, which do not make use of general factorization results, can be found in \cite{Casejuane}. 
\end{remark}

The next two results involve the coefficients $c_2$ and $c_3$ of the Conway polynomial. 
\begin{proposition} \label{c_2}
Let $K$ be a framed oriented knot. We have 
$$C_K\Big[ \dessin{0.5cm}{D21_1}\Big] = \frac{1}{8} fr(K)^2 + \frac{1}{24} - c_2(K),$$
$$C_K\Big[ \dessin{0.5cm}{D21_2}\Big] = c_2(K) - \frac{1}{24}.$$
\end{proposition}
\begin{proof}[Proof]
The fact that $C\Big[ \dessin{0.5cm}{D21_1}\Big]$ and $C\Big[ \dessin{0.5cm}{D21_2}\Big]$ define knot invariants follows from the Invariance Lemma \ref{Critere Invariance}, noting that the only $4T$ relation involving either of these two diagrams is a trivial one. \\
Let us prove the first statement. 
Let $K_+$, $K_-$ and $L_0 = K_1 \cup K_2$ be a skein triple at a knot crossing $c$. 
On one hand, by  the Skein Lemma \ref{Obvious lemma}, 
\begin{eqnarray*}
C_{K_+}\Big[ \dessin{0.5cm}{D21_1}\Big] - C_{K_-}\Big[ \dessin{0.5cm}{D21_1}\Big] 
 & = & C\Big[ \dessin{0.5cm}{D21_1}\Big] \Big(\dessin{0.9cm}{X1} \Big) \\
 & = & C\Big[ \dessin{0.5cm}{D11} \, \dessin{0.5cm}{D01}\Big] \Big(\dessin{0.9cm}{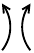} \Big) + 	
           C\Big[ \dessin{0.5cm}{D01} \, \dessin{0.5cm}{D11}\Big] \Big(\dessin{0.9cm}{X} \Big)\\
 & = & \frac{1}{2} fr(K_1) + \frac{1}{2} fr(K_2) \\
 & = & \frac{1}{2} \Big(fr(K_-) + 1 - 2 c_1(L_0)\Big).
\end{eqnarray*}
Here, the second equality is given by smoothing the chord contributed by $c$. As illustrated by Figure \ref{fig:lissage}, this smoothing maps $\dessin{0.65cm}{D21_1}$ to either 
$\dessin{0.65cm}{D11} \, \dessin{0.65cm}{D01}$ or $\dessin{0.65cm}{D01} \, \dessin{0.65cm}{D11}$; 
conversely, the latter two diagrams can only be obtained, by smoothing an internal chord, from $\dessin{0.65cm}{D21_1}$. 
The third equality then follows from Lemma \ref{lem:D11}, while the last equality is easily verified. 
On the other hand, the difference $(\frac{1}{8} fr(K_+)^2 + \frac{1}{24} - c_2(K_+)) - (\frac{1}{8} fr(K_-)^2 + \frac{1}{24} - c_2(K_-))$ can be written, using the skein relation for the Conway coefficients, as 
$$ \frac{1}{8}(\underbrace{fr(K_+) - fr(K_-)}_{=2})(\underbrace{fr(K_+) + fr(K_-)}_{=2fr(K_-) + 2}) - c_1(L_0).  $$ 
This shows that both invariants in the first statement have same variation formula under a (knot) crossing change, 
and it only remains to check that they coincide on both $U_0$ and $U_+$. 
Using the formulas for $\hat{Z}(U_0)$ and $\hat{Z}(U_+)$ recalled in Section \ref{sec:Kontsevich}, we have
$$C_{U_0}\Big[ \dessin{0.5cm}{D21_1} \Big] = \frac{1}{48} + \frac{1}{48} = \frac{1}{24} = \frac{1}{8} fr(U_0)^2) + \frac{1}{24} - c_2(U_0)$$
and 
$$C_{U_+}\Big[ \dessin{0.5cm}{D21_1} \Big] = \frac{1}{48} + \frac{1}{48} + \frac{1}{8} = \frac{1}{6} = \frac{1}{8} fr(U_+)^2 + \frac{1}{24} - c_2(U_+),$$
which concludes the proof of the first statement. \\
The proof of the second statement uses the same skein triple and is very similar. 
The same argument, only using Lemma \ref{lem:D12} instead of Lemma \ref{lem:D11}, gives 
\begin{eqnarray*}
C_{K_+}\Big[ \dessin{0.5cm}{D21_2}\Big] - C_{K_-}\Big[ \dessin{0.5cm}{D21_2}\Big] 
 & = & C\Big[ \dessin{0.5cm}{D21_2}\Big] \Big(\dessin{0.9cm}{X1} \Big) \\
 & = & C\Big[ \dessin{0.5cm}{D12}\Big] \Big(\dessin{0.9cm}{X} \Big) \\
 & = & C_{L_0} \Big[ \dessin{0.5cm}{D12}\Big] \\
 & = & c_1(L_0), 
\end{eqnarray*}
which clearly coincides with the variation formula for $c_2 - \frac{1}{24}$. 
The statement then follows from the equalities $C_{U_+}\Big[ \dessin{0.5cm}{D21_2} \Big]  = C_{U_0}\Big[ \dessin{0.5cm}{D21_2} \Big] =  - \frac{1}{24}$.
\end{proof}
\begin{remark}\label{rem:Oka}
 The second statement of Proposition \ref{c_2} can be found in \cite[Prop.~4.4]{Okamoto1997}. 
 Our next result actually fixes a mistake in \cite[Prop.~4.6~(1)]{Okamoto1997}; likewise, \cite[Prop.~4.6~(2)]{Okamoto1997} is corrected in Theorem \ref{c_n}. 
\end{remark}
\begin{proposition} \label{c_3}
Let $L$ be a framed oriented $2$-component link. Then $c_3(L)$ is given by the formula 
$$C_L\Big[ \dessin{0.5cm}{D32_2} \Big]+ C_L\Big[ \dessin{0.5cm}{D32_4} \Big] + C_L\Big[ \dessin{0.5cm}{D32_4b} \Big] + C_L\Big[ \dessin{0.5cm}{D32_7} \Big]+ C_L\Big[ \dessin{0.5cm}{D32_7b} \Big]. $$
\end{proposition}
\begin{proof}[Proof]
Set 
$$X_3:=C\Big[ \dessin{0.5cm}{D32_2} \Big]+ C\Big[ \dessin{0.5cm}{D32_4} \Big] + C\Big[ \dessin{0.5cm}{D32_4b} \Big] + C\Big[ \dessin{0.5cm}{D32_7} \Big]+ C\Big[ \dessin{0.5cm}{D32_7b} \Big]. $$ 
The only non-trivial $4T$ relations involving one of the diagrams above are the following 
$$ \dessin{0.65cm}{D32_1}  - \dessin{0.65cm}{D32_2}  = \dessin{0.65cm}{D32_3}  - \dessin{0.65cm}{D32_4} = \dessin{0.65cm}{D32_3b}  - \dessin{0.65cm}{D32_4b}, $$
and the Invariance Lemma \ref{Critere Invariance} can then be used to show that $X_3$ is a link invariant. \\
Now, let $L_+ = K_1^+ \cup K_2^+$, $L_- = K_1^- \cup K_2^-$ and $L_0$ be a skein triple at a mixed crossing. 
By  the Skein Lemma \ref{Obvious lemma} we have 
\begin{eqnarray*}
 C_{L_+}\Big[ \dessin{0.5cm}{D32_2} \Big] - C_{L_-}\Big[ \dessin{0.5cm}{D32_2} \Big] 
   & = & C\Big[ \dessin{0.5cm}{D32_2} \Big] \Big(\dessin{0.9cm}{X1} + \frac{1}{24}\dessin{0.9cm}{X3}\Big) \\
   & = & C\Big[ \dessin{0.5cm}{D32_2} \Big] \Big(\dessin{0.9cm}{X1}\Big) + \frac{1}{24}, 
\end{eqnarray*}
where the second equality follows from the observation that the local configuration $\dessin{0.9cm}{X3}$ on two circles yields the diagram $\dessin{0.65cm}{D32_2}$. 
Since the other diagrams defining $X_3$ have $\le 2$ mixed chords, we thus have by  the Skein Lemma \ref{Obvious lemma} that $X_3(L_+) - X_3(L_-)$ is given by 
$$\left(C\Big[ \dessin{0.45cm}{D32_2} \Big] \!\!+\! C\Big[ \dessin{0.45cm}{D32_4} \Big] \!\!+\! C\Big[ \dessin{0.45cm}{D32_4b} \Big]  \!\!+\!  C\Big[ \dessin{0.45cm}{D32_7} \Big]  \!\!+\!  C\Big[ \dessin{0.45cm}{D32_7b} \Big]\right)\!\! \Big(\dessin{0.8cm}{X1}\Big) + \frac{1}{24}. $$
But each of the above diagrams has the property that, smoothing a mixed chord always yields $\dessin{0.65cm}{D21_2}$ and, conversely, the latter can only be obtain by such a smoothing from one of the above five diagrams. This shows that 
$$ X_3(L_+) - X_3(L_-) =  C\Big[ \dessin{0.5cm}{D21_2} \Big]\Big(\dessin{0.9cm}{X}\Big) + \frac{1}{24}. $$
Proposition \ref{c_2}, and the skein relation for $c_3$, then give 
$$X_3(L_+) - X_3(L_-) = c_2(L_0) = c_3(L_+) - c_3(L_-).  $$
The result follows since both $X_3$ and $c_3$ vanish on split links.
\end{proof} 
\begin{remark}\label{rem:inv_split}
Using the non-trivial $4T$ relations given at the beginning of this proof, and the Invariance Lemma \ref{Critere Invariance}, we actually have that 
$C\Big[ \dessin{0.5cm}{D32_2} \Big]+ C\Big[ \dessin{0.5cm}{D32_4} \Big] + C\Big[ \dessin{0.5cm}{D32_4b} \Big]$ and 
$C\Big[ \dessin{0.5cm}{D32_7} \Big]+ C\Big[ \dessin{0.5cm}{D32_7b} \Big]$ are themselves link invariants. 
In fact, the latter is easily identified using Propositions \ref{c_2} and \ref{Inflation et invariance}: we have 
$$ C_{K_1\cup K_2}\Big[ \dessin{0.5cm}{D32_7} \Big] +C_{K_1\cup K_2}\Big[ \dessin{0.5cm}{D32_7b} \Big] =  l_{1,2}\Big( c_2(K_1)+ c_2(K_2)-\frac{1}{12}\Big). $$
\noindent Observe that this formula coincides with $c_3(K_1\cup K_2) - U_3(K_1\cup K_2)$, a fact that will be widely generalized in Section \ref{sec:conwayKontsevich}. 
\end{remark}

The next result will also be needed later. We omit the proof, 
since one can be found in \cite[Prop.~4.1~(3)]{Okamoto1997}; see also \cite{Casejuane} for a proof using the techniques of the present paper.
\begin{proposition}\label{prop:lk3}
Let $L$ be a framed oriented $2$-component link.  We have
 $$ C_{L}\Big[ \dessin{0.5cm}{D32_1} \Big] + C_{L}\Big[ \dessin{0.5cm}{D32_2} \Big] = \frac{1}{6}l_{1,2}^3. $$
\end{proposition}

More generally, the techniques used in this section can be used to identify, in degree $\le 3$, \emph{all} invariants arising as coefficients of the Kontsevich integral. 
All  remaining formulas are direct applications of our factorization results. For example, 
the following is given by Lemma \ref{lem:D11} and Proposition \ref{Factorisation gonflage}: 
\begin{equation}\label{rem:deg3}
 C_{K_1\cup K_2}\Big[ \dessin{0.5cm}{D32_3} \Big] + C_{K_1\cup K_2}\Big[ \dessin{0.5cm}{D32_4} \Big] =  \frac{1}{4} fr_1 \times l^2_{1,2}.
\end{equation}

\subsection{Conway polynomial and the Kontsevich integral}\label{sec:conwayKontsevich}

We now generalize Proposition \ref{c_3}, by explicitly identifying all Conway coefficients in terms of the Kontsevich integral. 

Recall from Section \ref{sec:essentiels} that $\mathcal{E}^-(n)$ denotes the set of $\ominus$-essential diagrams on $n$ circles, which are all of degree $n+1$. 

For an integer $n\ge 1$, denote by $\mathcal{P}(n+1)$ the set of all diagrams obtained by an infection on an element of $\mathcal{E}^-(n)$. 
\begin{ex}\label{ex3}
Since $\mathcal{E}^-(1)=\{\dessin{0.5cm}{D21_2}\}$, we have $\mathcal{P}(2)=\{\dessin{0.5cm}{D32_7};\dessin{0.5cm}{D32_7b}\}$.
\end{ex}
\begin{theorem} \label{c_n}
Let $n\ge 2$. For any framed oriented $n$-component link $L$, we have 
 $$\sum_{D \in \mathcal{E}^-(n) \cup \mathcal{P}(n)} C_L[D] = c_{n+1}(L).$$
\end{theorem}
\begin{remark}\label{rem:conway}
The case $n=1$ is somewhat particular, as it involves a correction term. Indeed, first note that $\mathcal{P}(1)=\emptyset$. We have $\mathcal{E}^-(1)=\{\dessin{0.5cm}{D21_2}\}$, and Proposition \ref{c_2} tells us that for a knot $K$, 
$$C_K\Big[\dessin{0.5cm}{D21_2}\Big] = c_2(K) - \frac{1}{24}.$$
Observe also that the case $n = 2$ is given by Proposition \ref{c_3}, since 
$\mathcal{E}^-(2)=\{\dessin{0.5cm}{D32_2};\dessin{0.5cm}{D32_4};\dessin{0.5cm}{D32_4b}\}$ and $\mathcal{P}(2)=\{\dessin{0.5cm}{D32_7};\dessin{0.5cm}{D32_7b}\}$. 
\end{remark}
\begin{proof}[Proof of Theorem \ref{c_n}]
Denote by  $X_n$ the left-hand term in the statement, which decomposes as 
$$ X_n = \sum_{D \in \mathcal{E}^-(n)} C_L[D] + \sum_{D \in \mathcal{P}(n)} C_L[D].$$
We first prove that $X_n$ indeed is an invariant. Actually, we show that each of the above two sums defines an invariant, by induction on $n$.  
The case $n=2$ is obtained by combining Remarks \ref{rem:inv_split} and \ref{rem:conway}.
Proposition \ref{Inflations et Essentiels} ensures that any element of $\mathcal{E}^-(n+1)$ is obtained by inflation on an element of $\mathcal{E}^-(n)$, up to permutation of the circle labels. 
A straightforward argument, using the Invariance Lemma \ref{Critere Invariance}, then shows that $\sum_{D \in \mathcal{E}^-(n+1)} C_L[D]$ is an invariant.\footnote{The argument is in the same spirit as in the proof of Lemma \ref{Factorisation fr}, and discusses the possible types of $4T$ relations depending on whether they involve chord(s) created during the inflation; we leave the details as an exercice to the reader.} 
On the other hand, Proposition \ref{Inflation et invariance} ensures that $\sum_{D \in \mathcal{P}(n+1)} C_L[D]$ is also an invariant.\\
Let us now prove the desired equality. This is again done by induction on $n$, using Proposition \ref{c_3} as initial step.  
Suppose that the equality holds for some $n\ge 2$, and consider a skein triple $(L_+,L_-,L_0)$ at a mixed crossing. 
Note that, for $n\ge 3$, 
there is no essential diagram with $\ge 3$ mixed chords between two given circles, by Proposition \ref{Inflations et Essentiels}. 
Hence by  the Skein Lemma \ref{Obvious lemma}, we have 
\begin{eqnarray*}
X_{n+1}(L_+) - X_{n+1}(L_-) 
 & = & \sum\limits_{D \in\mathcal{E}^-(n+1) \cup \mathcal{P}(n+1)} C[D] \Big(\dessin{0.9cm}{X1}\Big). 
\end{eqnarray*}
By smoothing the mixed chord in the above equality, we obtain 
\begin{eqnarray*}
X_{n+1}(L_+) - X_{n+1}(L_-) 
 & = & \sum\limits_{D \in\mathcal{E}^-(n) \cup \mathcal{P}(n)} C[D] \Big(\dessin{0.9cm}{X}\Big) \\
 & = & c_n(L_0) \\
 & = & c_{n+1}(L_+) - c_{n+1}(L_-). 
\end{eqnarray*}
Here, the fact that sum runs over $D \in \mathcal{E}^-(n) \cup \mathcal{P}(n)$ is ensured by Lemma \ref{unpeuutilequandmeme} and Remark \ref{rem:lissage}. 
The second equality is then given by the induction hypothesis, while the third equality is simply the skein relation for Conway coefficients. 
This proves that the invariants $X_{n+1}$ and $c_{n+1}$ have same variation formula under a mixed crossing change. 
The equality then follows from the fact that both invariants vanish on geometrically split links. 
\end{proof}

As mentioned in Remark \ref{rem:Oka}, Theorem \ref{c_n} fixes a mistake in \cite[Prop.~4.6~(2)]{Okamoto1997}. 
More precisely, \cite[Prop.~4.6~(2)]{Okamoto1997} treats the case $n=3$, but only considers the sum of coefficients given by 
$\mathcal{E}^-(n)$, and omits the correction terms given by $\mathcal{P}(n)$. 
Actually, considering only the terms given by $\mathcal{E}^-(n)$ yields the invariant $U_n$ introduced in Definition \ref{def:U_n} in terms of the Conway polynomial and linking numbers:  
\begin{proposition} \label{U_n}
Let $n\ge 2$. For any framed oriented $n$-component link $L$, we have 
$$\sum_{D \in \mathcal{E}^-(n)} C_L[D] = U_{n+1}(L).$$
\end{proposition}
\begin{proof}[Proof]
The fact that $\sum_{D \in \mathcal{E}^-(n)} C_L[D]$ defines an invariant for all $n$ was already discussed in the previous proof. 
The equality is proved by induction. The case $n=2$ is given by Proposition \ref{c_2}, and by Theorem \ref{c_n} we have that 
\begin{eqnarray*}
\sum_{D \in \mathcal{E}^-(n)} C[D]  
   & = & c_{n+1} - \sum_{D \in \mathcal{P}(n)} C[D] \\
   & = & c_{n+1} - \sum_{i=1}^n \sum_{D \in \mathcal{P}_i(n)} C[D],  \\
\end{eqnarray*}
where $\mathcal{P}_i(n)$ denote all elements of $\mathcal{P}(n)$ where the unique circle with a single leg 
(coming from an infection on some diagram of $\mathcal{E}^-(n-1)$) is  labeled by $i$. 
Then Proposition \ref{Inflation et invariance} and the induction hypothesis give that 
$\sum_{D \in \mathcal{P}_i(n)} C[D] = \sum_{i\neq j} U_{n} l_{i,j}$, which concludes the proof. 
\end{proof}

\section{The $\mu_n$ invariants}\label{sec:mu_n}

In the rest of this paper, we will make use of the following. 
\begin{notation}
For a chord diagram $D$ on $n$ circles, we denote by $\iota_\Theta(D)$ the rational coefficient such that 
  $$ \left(\iota_1(D \# \nu^n)\right)_1=\iota_\Theta(D)\times \dessin{0.5cm}{T}. $$
For a framed oriented $n$-component link $L$, we denote by 
  $$  \mathcal{C}_L[D] := C_L[D] \times \iota_\Theta(D). $$
In other words, $\mathcal{C}_L[D]$ is the contribution to $\Big( \iota_1(\check{Z}(L)) \Big)_1$ of a diagram $D$ in the Kontsevich integral $\hat{Z}(L)$, 
and $\iota_\Theta(D)$ is the part of this contribution that comes from the $\iota_1$ map (and the normalization by copies of $\nu$'s) on this diagram, while $C_L[D]$ is the part that comes from the Kontsevich integral itself. 
\end{notation}

\begin{ex} \label{footnote}
It follows from the definitions of $\check{Z}$ and $\iota_1$ that
$\iota_\Theta \Big(\dessin{0.5cm}{D21_1}\Big)  = \frac{1}{6}$, $\iota_\Theta \Big(\dessin{0.5cm}{D21_2}\Big) = -\frac{1}{3}$ 
and $\iota_\Theta \Big(\dessin{0.5cm}{D01}\Big) = \frac{1}{48}$. 
In this latter computation, the coefficient of $\dessin{0.5cm}{T}$ comes from the copy of $\nu$ added to the circle component. 
We stress that this type of contribution to the degree $1$ part of the LMO invariants, coming from the renormalization $\check{Z}$ of the Kontsevich integral,  only occurs with the trivial diagram $\dessin{0.5cm}{D01}$ on a single circle. 
In other words, for \emph{any} chord diagram $D\neq \dessin{0.5cm}{D01}$, we have  that 
$\left(\iota_1(D)\right)_1=\iota_\Theta(D)\times \dessin{0.5cm}{T}$. 
\end{ex}

The following is the main ingredient in our surgery formula for the Casson-Walker-Lescop invariant. 
Recall that $\mathcal{E}(n)$ denotes the set of all essential diagrams on $n$ circles. 
\begin{definition}\label{def:mun}
For all integers $n\ge 1$, let $\mu_n$ be the framed oriented $n$-component link invariant defined by  
 $$ \mu_1(K) = 2 \sum_{D\in \{\dessin{0.3cm}{D01}\}\cup \mathcal{E}(1)} \mathcal{C}_L[D] $$
and for all $n\ge 2$, 
 $$ \mu_n(L) = 2 \sum_{D\in \mathcal{E}(n)} \mathcal{C}_L[D]. $$
\end{definition}

Our task is now to make this definition completely explicit. 

\begin{remark}
The fact that the above formula indeed defines a link invariant is not completely obvious (in particular, this is not a mere application of the Invariance Lemma \ref{Critere Invariance}); we postpone the justification to Remark \ref{rem:jesuisinvariant} at the end of this section. 
\end{remark}

\subsection{Cases $n=1$ and $2$}

We listed in Example \ref{ex:essentiels} all essential diagrams on $1$ or $2$ circles, and we can thus describe $\mu_1$ and $\mu_2$ explicitly. 
\begin{lemma}\label{ex:mu1}
For a framed oriented knot $K$, we have 
 $$ \mu_1(K) = \frac{1}{24} fr(K)^2 - c_2(K) + \frac{1}{12}.$$ 
\end{lemma}
\begin{proof}[Proof]
We know that $\mathcal{E}^+(1) = \{ \dessin{0.5cm}{D21_1} \}$ and $\mathcal{E}^-(1) = \{ \dessin{0.5cm}{D21_2} \}$, and moreover we saw in Example \ref{footnote} that 
$\iota_\Theta \Big(\dessin{0.5cm}{D21_1}\Big)  = \frac{1}{6}$, $\iota_\Theta \Big(\dessin{0.5cm}{D21_2}\Big) = -\frac{1}{3}$ 
and $\iota_\Theta \Big(\dessin{0.5cm}{D01}\Big) = \frac{1}{48}$. 
Hence by Proposition \ref{c_2}, the invariant $\mu_1$ for a framed knot $K$ is given by 
\begin{eqnarray*}
 \mu_1(K) & = & 2\Big( \frac{1}{48} C_K\Big[\dessin{0.5cm}{D01}\Big] + \frac{1}{6} C_K\Big[\dessin{0.5cm}{D21_1}\Big] - \frac{1}{3} C_K\Big[\dessin{0.5cm}{D21_2}\Big]\Big) \\
  & = &  \frac{1}{24} + \frac{1}{3} \left(C_K\Big[\dessin{0.5cm}{D21_1}\Big]+C_K\Big[\dessin{0.5cm}{D21_2}\Big]\right) 
          - C_K\Big[\dessin{0.5cm}{D21_2}\Big] \\
  & = & \frac{1}{24} + \frac{1}{24} fr(K)^2 - U_2(K),
\end{eqnarray*}
where the last equality uses Propositions \ref{c_2} and \ref{U_n}. 
The result then follows from the definition of $U_2$ (Definition \ref{def:U_n}). 
\end{proof}
\begin{lemma}\label{ex:mu2}
For a framed oriented $2$-component link $L=K_1\cup K_2$, we have 
 $$ \mu_2(L)=\frac{1}{12} l_{1,2}^3 + \frac{f_1 + f_2}{12} l_{1,2}^2 - c_3(L) + l_{1,2}\Big( c_2(K_1) + c_2(K_2) - \frac{1}{12} \Big). $$ 
\end{lemma}
\begin{proof}[Proof]
We have, up to permutation of the circle labels,  
$$\mathcal{E}^+(2) = \{ \dessin{0.5cm}{D32_1} \, ; \, \dessin{0.5cm}{D32_3} \}\, \textrm{ and } \,\mathcal{E}^-(2) = \{ \dessin{0.5cm}{D32_2}\,;\,\dessin{0.5cm}{D32_4} \}.$$
Moreover, we have that 
$\iota_\Theta \Big(\dessin{0.5cm}{D32_1}\Big) = \frac{1}{4}$,  
$\iota_\Theta \Big(\dessin{0.5cm}{D32_3}\Big) = \frac{1}{6}$, 
$\iota_\Theta \Big(\dessin{0.5cm}{D32_2}\Big) = -\frac{1}{4}$ and  $\iota_\Theta \Big(\dessin{0.5cm}{D32_4}\Big) = -\frac{1}{3}$.   
Thus, for a $2$-component link $L=K_1\cup K_2$, we have 
\begin{eqnarray*}
 \mu_2(L) & = & 2\Big( \frac{1}{4} C_L\Big[\dessin{0.5cm}{D32_1}\Big]\!\! -\! \frac{1}{4} C_L\Big[\dessin{0.5cm}{D32_2}\Big]\!\!
          +\! \frac{1}{6} C_L\Big[\dessin{0.5cm}{D32_3}\Big]\!\! -\! \frac{1}{3} C_L\Big[\dessin{0.5cm}{D32_4}\Big]\Big) \\
  & = & \frac{1}{2} \left(C_L\Big[\dessin{0.5cm}{D32_1}\Big]\!\! +\!C_L\Big[\dessin{0.5cm}{D32_2}\Big]\right) \!
            + \! \frac{1}{3} \left(C_L\Big[\dessin{0.5cm}{D32_3}\Big]\!\!+\!C_L\Big[\dessin{0.5cm}{D32_4}\Big]\right) \\
  &    &   - \left( C_L\Big[\dessin{0.5cm}{D32_2}\Big] \!\!+\! C_L\Big[\dessin{0.5cm}{D32_4}\Big]\right) \\
  & = & \frac{1}{12} l_{1,2}^3 + \frac{f_1 + f_2}{12} l_{1,2}^2 - U_3(L). 
\end{eqnarray*}
Here, the final equality uses Proposition \ref{prop:lk3}, the formula given in (\ref{rem:deg3}), and Proposition \ref{U_n}. 
Hence from Definition \ref{def:U_n} we obtain the desired formula.
\end{proof}
\subsection{General case}\label{ex:muk}
Let us now investigate the invariant $\mu_n$ for $n\ge 3$. 

As pointed out in Remark \ref{rem:seeds}, all essential diagrams are obtained by iterated inflations from a few basic diagrams, 
namely either $\dessin{0.65cm}{D21_1}$ and $\dessin{0.65cm}{D32_1}$ for  $\oplus$-essential diagram,  
and $\dessin{0.65cm}{D21_2}$ and $\dessin{0.65cm}{D32_2}$ for $\ominus$-essential ones. 
Note that, in each of these four diagrams, the role of all chords is completely symmetric. 
We can thus define four families of \emph{unordered} chord diagrams,\footnote{A chord diagram is unordered if we do not specify an order on the circle components. }
$D_+(a,b)$, $ D_-(a,b)$, $D_+(a,b,c)$ and $D_-(a,b,c)$,  
$a,b,c\in \mathbb{N}$, as follows.   
\begin{definition}\label{def:D}
For integers $a,b,c$, such that $a\ge b\ge c\ge 0$,  

$\bullet$ $D_+(a,b)$ is the unordered $\oplus$-essential diagram on $a+b+1$ circles obtained from $\dessin{0.65cm}{D21_1}$ by $a$ successive inflations on one chord, 
and $b$ inflations on the other chord, 

$\bullet$ $D_-(a,b)$ is the unordered $\ominus$-essential diagram on $a+b+1$ circles obtained in the same way from $\dessin{0.65cm}{D21_2}$, 

$\bullet$ $D_+(a,b,c)$ is the unordered $\oplus$-essential diagram on $a+b+c+2$ circles obtained from $\dessin{0.65cm}{D32_1}$ by respectively $a$, 
$b$ and $c$ successive inflations on each chord,  

$\bullet$ $D_-(a,b,c)$ is the unordered $\ominus$-essential diagram on $a+b+c+2$ obtained in the same way from $\dessin{0.65cm}{D32_2}$.
\end{definition}
Some examples are given below: 
\begin{center}
 \includegraphics[scale=1.2]{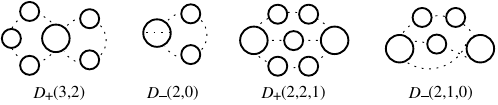}
\end{center}

We denote respectively by $\{D_+(a,b)\}$ and $\{D_-(a,b)\}$ the set of all chord diagrams obtained by ordering the circles of $D_+(a,b)$ and $D_-(a,b)$ from $1$ to $a+b+1$ 
in all possible ways. We also define $\{D_+(a,b,c)\}$ and $\{D_-(a,b,c)\}$ in a similar way.  Remark \ref{rem:seeds} can then be rephrased as the equality
\begin{equation}\label{eq:split}
\mathcal{E}^\pm(n+2)=\bigcup_{\substack{a+b=n+1 \\ a\ge b}} \{D_\pm(a,b)\} \cup \bigcup_{\substack{a+b+c=n \\ a\ge b\ge c}} \{D_\pm(a,b,c)\}.
\end{equation}
We also set 
$$ \mathcal{D}(a,b) =  \{D_+(a,b)\} \cup \{D_-(a,b)\} \quad\textrm{and} \quad \mathcal{D}(a,b,c) =  \{D_+(a,b,c)\} \cup \{D_-(a,b,c)\}. $$

We can identify explicitly the coefficients of such diagrams in the Kontsevich integral. This uses the following notation. 
\begin{notation}
Given two integers $i$ and $j$, and a set $I=\{i_1,\cdots,i_k\}$ of $k$ pairwise distinct integers, all different from $i$ and $j$. We set 
 $$ \mathcal{L}_{i,j,I} := \sum_{\sigma\in S_k} l_{i,i_{\sigma(1)}}\times l_{i_{\sigma(1)},i_{\sigma(2)}}\times \cdots \times l_{i_{\sigma(k-1)},i_{\sigma(k)}}\times l_{i_{\sigma(k)},j}.$$
We abbreviate $\mathcal{L}_{i,I}=\mathcal{L}_{i,i,I}$, and use the convention $\mathcal{L}_{i,j,\emptyset}=l_{i,j}$ if $i\neq j$, and $\mathcal{L}_{i,\emptyset}=fr_i$. 
\end{notation}
\begin{theorem}\label{thm:coeffs essentiels}
For all $a,b\in \mathbb{N}$ such that $a>0$ and $a\ge b$, 
$$
 \sum_{D\in \mathcal{D}(a,b)} C_L[D] = \frac{1}{4} \sum_{i=1}^{n} \sum_{\mathcal{I}_i(a,b)} \mathcal{L}_{i,I}  \mathcal{L}_{i,J} 
$$
where we sum over the set $\mathcal{I}_i(a,b)$ of all partitions $I\cup J=\{1,\cdots,a+b+1\}\setminus \{i\}$ such that $\vert I\vert=a$ and $\vert J\vert = b$. 

For all $a,b,c\in \mathbb{N}$ such that $a>0$ and $a\ge b\ge c$, 
$$
{ \everymath={\displaystyle}
 \sum_{D\in \mathcal{D}(a,b,c)} C_L[D] = \left\{
\begin{array}{ll}
\frac{1}{2} \sum_{1\le i< j\le n} l_{i,j}^2 \mathcal{L}_{i,j,\{1,\cdots,a+2\}\setminus \{i,j\}} &  \textrm{ if $b=c=0$, }\\[0.2cm]
\sum_{1\le i< j\le n} \sum_{\mathcal{I}_{i,j}(a,b,c)} \mathcal{L}_{i,j,I}  \mathcal{L}_{i,j,J}  \mathcal{L}_{i,j,K} & \textrm{ otherwise,}
\end{array} \right. 
}
$$
where the last sum is over the set $\mathcal{I}_{i,j}(a,b,c)$ of all partitions $I\cup J\cup K=\{1,\cdots,a+b+c+2\}\setminus \{i,j\}$ such that $\vert I\vert=a$,  $\vert J\vert = b$ and $\vert K\vert =c$. 
\end{theorem}
\begin{proof}[Proof]
In this proof, we call \emph{order $k$ chain}  ($k\ge 0$) the result of $k$ successive inflations on a chord; 
in particular, an order $1$ chain is an inflated chord, in the sense of  Proposition \ref{Factorisation gonflage}. \\
Let us focus on the first half of the statement, involving the diagrams $D_\pm(a,b)$. 
We first consider the case $a>0$ and $b=0$. 
The diagrams $D_+(a,0)$ and $D_-(a,0)$ are obtained by inserting, in all possible ways, an order $a$ chain to $\dessin{0.65cm}{D11}$. 
Such an insertion is achieved by, first, an infection, followed by $a-1$ iterated infections on the newly created circle, and finaly, 
the insertion of a chord between the newest and the initial circles. 
These operations endow $D_\pm(a,0)$ with a canonical ordering, for which Propositions \ref{Inflation et invariance} and \ref{Factorisation lk} if $a>1$ (resp. Proposition \ref{Factorisation gonflage} if $a=1$) 
ensure that $ C_L[D_+(a,0)]+C_L[D_-(a,0)]$ indeed is a link invariant, and is given by 
$$ C_L[D_+(a,0)]+C_L[D_-(a,0)] =  
  \frac{1}{2} fr_{1}\times l_{1,2}\times l_{2,3}\times
  \cdots \times l_{a+1,1}. $$
The desired formula is then obtained by considering all possible orders on $D_\pm(a,0)$, noting that, for symmetry reasons, each term appears twice in the defining sum for 
$\mathcal{L}_{i,j,\{1,\cdots,a+1\}\setminus \{i,j\}}$ when $i=j$, hence an extra $\frac{1}{2}$ factor. \\
In the case where $a>0$ and $b>0$, the diagrams  $D_\pm(a,b)$ are the result of inserting on $\dessin{0.65cm}{D01}$, in all possible ways, an order $a$ chain, followed by an order $b$ chain.
The exact same argument then applies. \\
The second half of the statement is proved in a strictly similar way. 
The first case uses the fact that the diagrams $D_\pm(a,0,0)$ ($a\ge 1$) are obtained by inserting, 
in all possible ways, an order $a$ chain to $\dessin{0.65cm}{D22}$ (thus using Lemma \ref{lem:D22}). 
Likewise, for the second case, $D_\pm(a,b,c)$ ($a,b\ge 1$) is obtained by inserting three chains of order $a$, $b$ and $c$ to the empty diagram on two circle. 
\end{proof}

Using Theorem \ref{thm:coeffs essentiels}, we can give the desired explicit formula for the invariants $\mu_n$, 
for any $n\ge 3$, in terms of Conway coefficients and the linking matrix.  
\begin{theorem}\label{cor:mun}
For all $n\ge 1$, and for any framed oriented $(n+2)$-component link $L$, we have 
\begin{eqnarray*}
\mu_{n+2}(L)    & = &  \frac{1}{12} \sum_{\substack{1\le i\le n \\ a+b=n+1 \\ a\ge b}} \sum_{\mathcal{I}_i(a,b)} \mathcal{L}_{i,I}  \mathcal{L}_{i,J}  
     + \frac{1}{2} \sum_{\substack{1\le i< j\le n \\ a+b+c=n \\ a\ge b\ge c}}  \sum_{\mathcal{I}_{i,j}(a,b,c)} \mathcal{L}_{i,j,I}  \mathcal{L}_{i,j,J}  \mathcal{L}_{i,j,K} \\
    &  & +  \frac{1}{4} \sum_{1\le i< j\le n} l_{i,j}^2 \mathcal{L}_{i,j,\{1,\cdots,n+2\}\setminus \{i,j\}} - U_{n+3}(L). 
\end{eqnarray*}
\end{theorem}
\begin{proof}[Proof]
According to (\ref{eq:split}), we have 
 $$\mu_{n+2}(L)  = 2 \sum_{\substack{a+b=n+1 \\ a\ge b}}  \sum_{D\in \mathcal{D}(a,b)} \mathcal{C}_L[D] 
 + 2\sum_{\substack{a+b+c=n \\ a\ge b\ge c}}  \sum_{D\in \mathcal{D}(a,b,c)} \mathcal{C}_L[D]. $$ 
By the definition of the $\iota_1$ map, it is  easily verified that for any $a,b,c\ge 0$, we have  
$\iota_\Theta(D_+(a,b)) = \frac{1}{6}$, $\iota_\Theta(D_-(a,b)) = -\frac{1}{3}$, $\iota_\Theta(D_+(a,b,c)) = \frac{1}{4}$ and $\iota_\Theta(D_-(a,b,c))$ $= -\frac{1}{4}$. 
Hence, recalling that $ \mathcal{D}(a,b) =  \{D_+(a,b)\} \cup \{D_-(a,b)\}$ and $\mathcal{D}(a,b,c) =  \{D_+(a,b,c)\} \cup \{D_-(a,b,c)\}$, 
 we have 
\begin{eqnarray*}
\mu_{n+2}(L) & = &\, 2  \sum_{\substack{a+b=n+1 \\ a\ge b}}  \Big( \frac{1}{6}\sum_{D\in \{D_+(a,b)\}} C_L[D]  -\frac{1}{3} \sum_{D\in \{D_-(a,b)\}} C_L[D] \Big) \\
    &    & + 2\sum_{\substack{a+b+c=n \\ a\ge b\ge c}}  \Big( \frac{1}{4} \sum_{D\in \{D_+(a,b,c)\}} C_L[D]  -\frac{1}{4} \sum_{D\in \{D_-(a,b,c)\}} C_L[D] \Big).
\end{eqnarray*}
It follows that 
\begin{eqnarray*} 
\mu_{n+2}(L) 
    & = & \sum_{\substack{a+b=n+1 \\ a\ge b}}  \Big( \frac{1}{3}\sum_{D\in \mathcal{D}(a,b)} C_L[D]  - \sum_{D\in \{D_-(a,b)\}} C_L[D] \Big) \\
    &    & + \sum_{\substack{a+b+c=n \\ a\ge b\ge c}}  \Big( \frac{1}{2} \sum_{D\in \mathcal{D}(a,b,c)} C_L[D]  - \sum_{D\in \{D_-(a,b,c)\}} C_L[D] \Big)  \\
    & = & \frac{1}{3} \sum_{\substack{a+b=n+1 \\ a\ge b}}  \sum_{D\in \mathcal{D}(a,b)}\!\!\!\!\!\! C_L[D] 
           + \frac{1}{2} \sum_{\substack{a+b+c=n \\ a\ge b\ge c}}  \sum_{D\in \mathcal{D}(a,b,c)}\!\!\!\!\!\! C_L[D]  -\!\!\!\!\!\! \sum_{D\in \mathcal{E}^-(n+2)} C_L[D]
\end{eqnarray*}
where the last equality uses (\ref{eq:split}). 
It only remains to use Theorem \ref{thm:coeffs essentiels} to express the first two terms in terms of linkings and framings, 
and Proposition \ref{U_n} to identify the last sum with $U_{n+3}$.
\end{proof}

\begin{remark}\label{rem:jesuisinvariant}
The fact that the defining formula for $\mu_{n+2}$ gives a link invariant follows readily from the decomposition 
$$ \mu_{n+2}(L) =\frac{1}{3} \sum_{\substack{a+b=n+1 \\ a\ge b}}  \sum_{D\in \mathcal{D}(a,b)}\!\!\!\!\!\! C_L[D] 
           + \frac{1}{2} \sum_{\substack{a+b+c=n \\ a\ge b\ge c}}  \sum_{D\in \mathcal{D}(a,b,c)}\!\!\!\!\!\! C_L[D]  -\!\!\!\!\!\! \sum_{D\in \mathcal{E}^-(n+2)} C_L[D] $$
and the fact that each of the above three sums defines a link invariant, by Theorem \ref{thm:coeffs essentiels} and Proposition \ref{U_n}. 
\end{remark}
\medskip

The techniques used to show Theorem \ref{thm:coeffs essentiels} can also be used to prove the following technical result. 
 Recall from Definition \ref{rem:chain} that a chain of $m$ circles is a connected chord diagram on $m$ circles with two legs on each circle. 
\begin{lemma}\label{a_la_chain}
 Let $C_m$ be a chain of $m$ circles, and let $I=\{i_1,\cdots,i_m\}$ be a set of $m$ pairwise distinct indices. 
 Let $\mathcal{D}(I)$ be the set of all chord diagrams obtained by labeling $C_m$ by the elements of $I$, in all possible ways. 
 Then for a framed oriented $m$-component link $L$, we have 
   $$ \sum_{D\in \{D(I)\}} C_L[D] = 
   \begin{cases} 
   \frac{1}{2} fr_{i_1} & \text{if } m=1 \\ 
   \frac{1}{2}\sum_{\sigma \in S_{m-1}} l_{i_m,i_{\sigma(1)}} l_{i_{\sigma(1)}, i_{\sigma(2)}} \times \cdots \times  l_{i_{\sigma(m-1)}, i_m} & \text{if } m>1 
   \end{cases}.$$
\end{lemma}
\begin{proof}[Proof]
 If $m=1$ or $2$, then there is a unique labeling of $C_I$ and the result is given by Lemmas \ref{lem:D11} and  \ref{lem:D22}, respectively. 
 If $m>2$,  an element of $\mathcal{D}(I)$ can be seen as obtained from $\dessin{0.5cm}{D01}$, labeled by $i_m$, by adding an order $m-1$ chain of circles, labeled by $i_1,\cdots, i_{m-1}$ in all possible ways. 
 The same arguments as in the proof of Theorem \ref{thm:coeffs essentiels} then give the desired formula.  
\end{proof}

\section{Surgery formula for the Casson-Walker-Lescop invariant}

We now prove the surgery formula stated in the introduction.

\subsection{Setup}

In the previous sections, we identified certain combinations of coefficients of the Kontsevich integral in terms of classical invariants. 
In order to derive from these results a formula for the Casson-Walker-Lescop invariant, we now have to study how these particular diagrams contribute to the degree $\le 1$ part of the LMO invariant. 
Recall indeed that 
$$Z_1^{LMO}(S^3_L) = \left( \frac{\iota_1(\check{Z}(L))}{\iota_1(\check{Z}(U_+))^{\sigma_+(L)} \iota_1(\check{Z}(U_-))^{\sigma_-(L)}}\right)_{\le 1}, $$
and that the coefficient of $\dessin{0.5cm}{T}$ in $Z_1^{LMO}(S^3_L)$ is $\frac{(-1)^{\beta_1(S^3_L)}}{2} \lambda_L(S^3_L)$ (Theorem \ref{LMOCasson}). 
By Equation (\ref{Denominateur LMO}), there are two types of contributions to the coefficient of $\dessin{0.5cm}{T}$ coming from this formula:
\begin{enumerate}
	\item The diagram $\dessin{0.5cm}{T}$ comes from the denominator with coefficient $\frac{(-1)^{\sigma_+(L)} \sigma(L)}{16}$, and is multiplied by a constant term coming from $\iota_1(\check{Z}(L))$. 
	\item The diagram $\dessin{0.5cm}{T}$ comes from $\iota_1(\check{Z}(L))$, with some coefficient, and is multiplied by the coefficient  $(-1)^{\sigma_+(L)}$ coming from the denominator.
\end{enumerate}
Summarizing, we have the following key equality
\begin{equation}\label{eq:main}
\frac{(-1)^{\beta_1(S^3_L)}}{2} \lambda_L(S^3_L) =  \frac{(-1)^{\sigma_+(L)} \sigma(L)}{16} \Big( \iota_1(\check{Z}(L)) \Big)_0 + (-1)^{\sigma_+(L)} \times \Big( \iota_1(\check{Z}(L)) \Big)_1. \tag{*}
\end{equation}

\subsection{The surgery formula}

Recall from Section \ref{sec:conv} that, if $\mathbb{L}$ is the linking matrix of a framed oriented $n$-component link, 
and if $I$ is some subset of $\{1,\cdots ,n\}$, we denote by $\mathbb{L}_{\check{I}}$ the matrix obtained from $\mathbb{L}$ by 
deleting the lines and column indexed by elements of $I$. 

\begin{theorem} \label{Thm general}
Let $L$ be a framed oriented $n$-component link in $S^3$ with linking matrix $\mathbb{L}$. 
Let $S^3_L$ be the result of surgery on $S^3$ along $L$. 
The  Casson-Walker-Lescop invariant $\lambda_L(S^3_L)$  is given by 
$$ \frac{(-1)^{\sigma_-(L)} \det \mathbb{L}}{8} \sigma(L) 
+ (-1)^{n+\sigma_-(L)}\sum_{k=1}^{n} \sum_{\substack{I \subset \{ 1, \ldots, n \} \\ |I| = k}} (-1)^{n - k} \det \mathbb{L}_{\check{I}} \mu_k(L_I).$$
\end{theorem}
 Some remarks are in order. 
\begin{remark}\label{rem:lescoop}
As pointed out in the introduction, this recovers Lescop's third formula \cite[Prop.~1.7.8]{Lescop} 
for her extension of the Casson-Walker invariant. 
In particular, the two formulas in Theorem \ref{thm:coeffs essentiels}, which underly the definition of the invariant $\mu_k$ by Theorem \ref{cor:mun}, 
correspond to the products of linkings $\Theta_b$ in \cite[Nota.~1.7.5]{Lescop}. More precisely, in the terminology of \cite[Fig.~1.2]{Lescop}, 
the first formula corresponds to $\Theta_b$ in the case of a \lq Figure-eight graph\rq, while the second formula corresponds to the case of a \lq beardless $\Theta$\rq. 
\\
It is quite interesting to see how Lescop's \lq chain products of linking numbers\rq \, $\Theta_b$ 
appear naturally in our proof from the combinatorics of chord and Jacobi diagrams and the universal Kontsevich-LMO invariants. 
\end{remark} 
\begin{remark}\label{rem:cestcadeaucamfaitplaisir}
As an illustration, let us focus on the case $n=2$ for rational homology spheres.
Let $L = K_1 \cup K_2$ be a framed oriented link whose linking matrix $\mathbb{L} = \left( \begin{smallmatrix}
a & n\\
n & b	
\end{smallmatrix} \right)$ has nonzero determinant. 
Then $S^3_L$ is a rational homology sphere and 
$\lambda_L(M) = \frac{1}{2}\vert \det \mathbb{L}\vert \lambda_W(M)$.
One can easily check that $(-1)^{\sigma_-(L)}$ is just the sign of $\det \mathbb{L}$, and 
Theorem \ref{Thm general} thus gives us 
$$\frac{1}{2} \det \mathbb{L} \lambda_W(S^3_L) = \frac{\det \mathbb{L}}{8} \sigma(L) 
+ \Big( \mu_2(L) - a\mu_1(K_2) - b\mu_1(K_1) \Big). $$
Using the explicit formulas for $\mu_1$ and $\mu_2$ given in Lemmas \ref{ex:mu1} and \ref{ex:mu2}, we then obtain the following formula  for 
$\frac{\det \mathbb{L}}{2}\left(\lambda_W(M) - \frac{1}{4} \sigma(L)\right)$: 
$$ a c_2(L_2) + bc_2(L_1) +\frac{n^3-n}{12} + \frac{(a+b)}{24}(2n^2-ab-2) - c_3(L) + n\left(c_2(L_1)+c_2(L_2)\right).$$
This recovers a result of S.~Matveev et M.~Polyak \cite[Thm.~6.3]{Matveev-Polyak}. 
\end{remark}
\medskip 

The rest of this section is devoted to the proof of Theorem \ref{Thm general}.

Recall that $\beta_1(S^3_L)$ is the nullity of $\mathbb{L}$, so that multiplying Equation  (\ref{eq:main}) by $2(-1)^{\beta_1(S^3_L)}$ gives 
$$ 
\lambda_L(S^3_L) =  \frac{(-1)^{\sigma_-(L)}}{8} \Big( \iota_1(\check{Z}(L)) \Big)_0 + (-1)^{n+\sigma_-(L)} \times 2\Big( \iota_1(\check{Z}(L)) \Big)_1.   
$$
Hence we are left with the explicit computations of $ \Big( \iota_1(\check{Z}(L)) \Big)_0$ and $ \Big( \iota_1(\check{Z}(L)) \Big)_1$. This is done in the following two lemmas. 

\begin{lemma}\label{eq1}
$$ \Big( \iota_1(\check{Z}(L)) \Big)_0 =(-1)^n \det \mathbb{L}.$$
\end{lemma}
\begin{proof}[Proof]
The diagrams in the Kontsevich integral of $L$ that contribute to $ \Big( \iota_1(\check{Z}(L)) \Big)_0$ are those that close into a constant, 
that is, disjoint unions of chains of circles.\footnote{Note indeed that the normalization in $\check{Z}(L)$ adding a copy of $\nu$ to each circle does not affect the degree $0$ part of $\hat{Z}(L)$. }
As pointed out in Section \ref{sec:essentiels}, chains of circles always close into the constant $(-2)$. 
The coefficient of a chain of $k$ circles in the Kontsevich integral is given in terms of coefficients of the linking matrix  
by Lemma \ref{a_la_chain}, and yields the following: 
$$\Big( \iota_1(\check{Z}(L)) \Big)_0 =\sum_{\{ I_1, \ldots, I_k \} \, \text{partition of}\, \{ 1, \ldots, n \}} (-1)^k \prod_{j=1}^{k} \mathcal{I}(I_j), $$
where 
$\mathcal{I}(I_j) = \mathcal{L}_{i_m,I_j\setminus \{i_m\}}=\sum_{\sigma \in S_{m-1}} l_{i_m,i_{\sigma(1)}} l_{i_{\sigma(1)}, i_{\sigma(2)}} \times \cdots \times  l_{i_{\sigma(m-1)}, i_m}$ 
if $I_j=\{i_1,\cdots,i_m\}$ with $m>1$, and $\mathcal{I}(I_j) = fr_{i_1}$ otherwise. 
We leave it as an exercice to the reader to check that this indeed gives $(-1)^n \det \mathbb{L}$.
\end{proof}

\begin{lemma}\label{eq2}
$$ \Big( \iota_1(\check{Z}(L)) \Big)_1 = \frac{1}{2}\sum_{k=1}^{n} \sum_{\substack{I \subset \{ 1, \ldots, n \} \\ |I| = k}} (-1)^{n - k} \det \mathbb{L}_{\check{I}} \mu_k(L_I).$$
\end{lemma}
\begin{proof}[Proof]
Computing $\Big( \iota_1(\check{Z}(L)) \Big)_1$ amounts to counting those diagrams in $\check{Z}(L)$ that close into $\dessin{0.5cm}{T}$. 
As observed in Example \ref{footnote}, 
a copy of $\dessin{0.5cm}{D01}$ in $\hat{Z}(L)$ yields such a term when adding a copy of $\nu$ in $\check{Z}(L)$, 
and this is the only contribution arising from this normalization $\check{Z}$. 
Hence a diagram from $\hat{Z}(L)$ that contributes to $\Big( \iota_1(\check{Z}(L)) \Big)_1$ 
is, for some $k$ such that $1\le k\le n$, a disjoint union of 
\begin{itemize}
\item a chord diagram on $(n-k)$ circles which is a union of chains of circles, which contributes by a constant, 
\item an element of $\mathcal{E}(k)$ if $k>1$, and either an element of $\mathcal{E}(1)$ or a copy of $\dessin{0.5cm}{D01}$ if $k=1$, which contributes by a $\dessin{0.5cm}{T}$ with some coefficient.
\end{itemize}
For a subset $I$ of $k>1$ elements of $\{1,\cdots,n\}$, 
the contribution to $\Big( \iota_1(\check{Z}(L)) \Big)_1$ of all diagrams in $\mathcal{E}(k)$, labeled by $I$ in all possible ways, is given by $\frac{1}{2}\mu_k(L_I)$ by virtue of  Definition \ref{def:mun}; on the other hand, 
the proof of Lemma \ref{eq1} above tells us that the contribution of all possible unions of chains of circles labeled by $\{1,\cdots,n\}\setminus I$ 
is precisely $(-1)^k \det \mathbb{L}_{\check{I}}$. The same holds for $k=1$, noting the change in the formula for $\mu_1$ given in Definition \ref{def:mun}.
The formula follows, by taking the sum over all possible subsets $I$. 
\end{proof}


\bibliographystyle{abbrv}
\bibliography{casson}

\end{document}